\documentclass{amsart}

\usepackage{amsfonts}
\usepackage{graphicx}
\usepackage{epstopdf}
\usepackage{algorithmic}
\usepackage{comment}
\usepackage{setspace}
\usepackage[hidelinks]{hyperref}

\theoremstyle{definition}
\newtheorem{theorem}{Theorem}[section]
\newtheorem{corollary}[theorem]{Corollary}
\newtheorem{proposition}[theorem]{Proposition}
\newtheorem{lemma}[theorem]{Lemma}
\newtheorem{definition}[theorem]{Definition}
\newtheorem{notation}[theorem]{Notation}
\newtheorem{remark}[theorem]{Remark}
\newtheorem{assumption}{Assumption}

\newcommand{\matdev}{\partial^{\bullet}}

\newcommand{\ddt}[1]{\frac{\partial #1}{\partial t}}
\newcommand{\gradg}{\nabla_{\Gamma}}
\newcommand{\lapg}{\Delta_{\Gamma}}

\newcommand{\divg}{\gradg \cdot}
\newcommand{\Ech}{E^{\text{CH}}}
\newcommand{\Echd}{E^{\text{CH},\delta}}
\newcommand{\grad}{\nabla}

\def\Xint#1{\mathchoice
	{\XXint\displaystyle\textstyle{#1}}%
	{\XXint\textstyle\scriptstyle{#1}}%
	{\XXint\scriptstyle\scriptscriptstyle{#1}}%
	{\XXint\scriptscriptstyle\scriptscriptstyle{#1}}%
	\!\int}
\def\XXint#1#2#3{{\setbox0=\hbox{$#1{#2#3}{\int}$ }
		\vcenter{\hbox{$#2#3$ }}\kern-.6\wd0}}

\def\dashint{\Xint-}
\newcommand{\mval}[2]{\dashint_{#2} #1}

\newcommand{\mbf}[1]{\mathbf{#1}}
\newcommand{\mbb}[1]{\mathbb{#1}}

\newcommand{\invoperator}[1]{\mathcal{S}_{#1}}
\newcommand{\esssup}[1]{\underset{#1}{\operatorname{ess \ sup \ }}}

\title[Evolving surface Cahn-Hilliard equation with a degenerate mobility]{The evolving surface Cahn-Hilliard equation with a degenerate mobility}
\author{Charles M. Elliott \and
	Thomas Sales}
\address{Mathematics Institute, Zeeman Building, University of Warwick, Coventry CV4 7AL, UK}
\email{\href{mailto:c.m.elliott@warwick.ac.uk}{c.m.elliott@warwick.ac.uk}, \href{mailto:tom.sales@warwick.ac.uk}{tom.sales@warwick.ac.uk}}
\date{}
\begin{document}

\begin{abstract}
We consider the existence of suitable weak solutions to the Cahn-Hilliard equation with a non-constant (degenerate) mobility on a class of evolving surfaces.
We also show weak-strong uniqueness for the case of a positive mobility function, and under some further assumptions on the initial data we show uniqueness for a class of strong solutions for a degenerate mobility function.
\end{abstract}

\maketitle
\section{Introduction}
This paper is concerned with the Cahn-Hilliard equation on a sufficiently smooth, closed evolving surface, $\Gamma(t) \subset \mbb{R}^3$ moving with prescribed velocity $V$, with a non constant mobility $M(\cdot)$,
\begin{align}
	\matdev u + (\gradg \cdot V)u = \gradg \cdot \left( M(u) \gradg \left(-\varepsilon \lapg u + \frac{1}{\varepsilon} F'(u)\right)\right), \label{cheqn}
\end{align}
as derived in \cite{caetano2021cahn}.
We note that we allow velocities $V$ which have some non-zero tangential component rather than evolution purely in the normal direction --- see also the discussion in \cite{caetano2021cahn}.
This equation is posed on the non-cylindrical space-time domain
\[ \mathcal{G}_T := \bigcup_{t \in [0,T]} \Gamma(t) \times \{t\}, \]
which is known since the velocity field $V$ is prescribed.
Here $\matdev$ denotes the material derivative following the flow of the velocity field $V$, $\gradg$ denotes the tangential gradient, and $\lapg$ denotes the Laplace-Beltrami operator --- all of which are defined in further detail in the following section.
In particular we are interested in the case where $M$ is a degenerate mobility, that is $M(r) > 0$ on $(-1,1)$ and $M(\pm1) = 0$ --- a typical example being $M(r) = 1 - r^2$.
We focus on potential functions of the form $F = F_1 + F_2$ where $F_1$ is convex (and possibly singular), and $F_2'$ has at most linear growth.
We provide explicit assumptions later on.\\

Interest in the Cahn-Hilliard equation posed on an evolving surface stems from applications, for instance \cite{eilks2008numerical,ErlAziKar01,GesMcCGas15,zhiliakov2021experimental} and the references therein.
This has resulted in work from both an analytic perspective, see for instance \cite{abels2024diffuse2,abels2024diffuse,caetano2023regularization,elliott2024navier}, and a computational perspective, see for instance \cite{beschle2022stability,elliott2015evolving,elliott2024fully}.
However, to the authors' knowledge there has been no rigorous study of the Cahn-Hilliard equation with a degenerate mobility posed on an (evolving) surface.
We note also that the models in \cite{eilks2008numerical,ErlAziKar01,GesMcCGas15} involve degenerate mobilities.\\

The main contributions of this work are twofold.
Firstly we extend the existing analysis to a variable, and then degenerate, mobility function, by adapting the approach of \cite{elliott1996cahn}.
Secondly, and the more novel part of this work, we show a weak-strong uniqueness result for a positive mobility function, and a uniqueness result for a class of sufficiently smooth solutions for the degenerate mobility.
For some existing further literature concerning the degenerate Cahn-Hilliard equation on a stationary Euclidean domain we refer to \cite{abels2013incompressible,dai2016weak,elbar2023degenerate}, and we refer the reader to the recent preprint \cite{ConGalGat24} in which the authors also show a uniqueness result for a positive mobility function.
We also note that our existence theory for a degenerate mobility does not apply to general smoothly evolving surfaces without some further assumptions on the potential and mobility functions.

\section{Preliminaries}
We assume $\Gamma(t)$ is a $C^3$ closed, orientable evolving surface, where there exist $C^3$ diffeomorphisms $\Phi(t) : \Gamma_0 \rightarrow \Gamma(t)$ such that for points $x_0 \in \Gamma_0$ one has
\[ \frac{d}{dt} \Phi(x_0,t) = V(\Phi(x_0,t),t), \]
for $V$ the velocity of $\Gamma(t)$.
We shall assume throughout that $V$ is sufficiently smooth so that
\[ \sup_{t \in [0,T]} \|V\|_{C^2(\Gamma(t))} \leq C < \infty. \]
We denote the pushforward of a function $\eta$ on $\Gamma_0$ as $\Phi_t \eta = \eta \circ \Phi(t)^{-1}$, and the pullback of a function $\psi$ on $\Gamma(t)$ by $\Phi_{-t} \psi = \psi \circ \Phi(t)$.
For a sufficiently smooth function $\eta$ the strong material derivative is defined as
\[ \matdev \eta := \Phi_t \left( \frac{d}{dt} \Phi_{-t} \eta \right), \]
and this is used to define suitable notion of a weak material derivative (see \cite{alphonse2023function,alphonse2015abstract}).\\

Denoting the normal vector on $\Gamma(t)$ by $\boldsymbol{\nu}$ one defines the tangential gradient as
\[ \gradg \eta := (I - \boldsymbol{\nu} \otimes \boldsymbol{\nu}) \nabla \eta^e,  \]
for $\eta^e$ denoting some differentiable extension of $\eta$ on a neighbourhood of $\Gamma(t)$.
Writing this componentwise as 
\[ \gradg \eta = \begin{pmatrix}
	\underline{D}_1 \eta\\
	\underline{D}_2 \eta\\
	\underline{D}_3 \eta
\end{pmatrix}, \]
one defines the tangential divergence of a vector field, $\boldsymbol{\eta}$ as
\[ \gradg \cdot \boldsymbol{\eta} = \sum_{i=1}^3 \underline{D}_i \boldsymbol{\eta}_i, \]
and the Laplace-Beltrami operator is defined as $\lapg \eta = \gradg \cdot \gradg \eta$.
It can be shown that these operators are independent of the choice of extension.

\subsection{Sobolev spaces}
\begin{definition}
	For $p \in [1, \infty]$, the Sobolev space $H^{1,p}(\Gamma)$ is then defined by
	$$ H^{1,p}(\Gamma) := \{ \eta \in L^p(\Gamma) \mid \underline{D}_i \eta \in L^p(\Gamma), \ i=1,...,n+1 \},$$
	and higher order spaces ($k \in \mathbb{N}$) are defined recursively by
	$$ H^{k,p}(\Gamma) := \{ \eta \in H^{k-1,p}(\Gamma) \mid \underline{D}_i \eta \in H^{k-1,p}(\Gamma), \ i = 1,...,n+1  \}, $$
	where $H^{0,p}(\Gamma) := L^p(\Gamma)$.
	These Sobolev spaces are known to be Banach spaces when equipped with norm,
	$$
	\| \eta \|_{H^{k,p}(\Gamma)} := \begin{cases}
		\left( \sum_{|\alpha| = 0}^{k}  \| \underline{D}^{\alpha} \eta \|_{L^p(\Gamma)}^p  \right)^{\frac{1}{p}}, & p \in [1, \infty),\\
		\max_{|\alpha|=1,...,k} \|\underline{D}^\alpha \eta \|_{L^\infty(\Gamma)}, & p = \infty,
	\end{cases}$$
	where we consider all weak derivatives of order $|\alpha|$.
	We use shorthand notation, $H^{k}(\Gamma) := H^{k,2}(\Gamma)$, for the case $p = 2$.
\end{definition}

Next we introduce some notation which will be used throughout.
\begin{notation}
	For a $\mathcal{H}^2 -$measurable set, $X \subset \mathbb{R}^{3}$, we denote the $\mathcal{H}^2$ measure of $X$ by
	$$ |X| := \mathcal{H}^2(X).$$
	
	For a function $\eta \in L^1(X)$ we denote the mean value of $\eta$ on $X$ by
	$$ \mval{\eta}{X} := \frac{1}{|X|} \int_X \eta .$$
\end{notation}

We refer the reader to \cite{aubin2012nonlinear,hebey2000nonlinear} for further details on Sobolev spaces defined on manifolds.

It can be shown that the pairs $(H^{m,p}(\Gamma(t)), \Phi_t)$ are compatible in the sense of \cite{alphonse2023function,alphonse2015abstract} for $m = 0,1,2$ and $p \in [1,\infty]$.
Compatibility of these spaces allows one to obtain Sobolev inequalities on $\Gamma(t)$ independent of $t$.\\

With these definitions, we can define time-dependent Bochner spaces.

\begin{definition}
	In the following we let $X(t)$ denote a Banach space dependent on $t$, for instance $H^{m,p}(\Gamma(t))$.
	The space $L^2_X$ consists of (equivalence classes of) functions
	\begin{gather*}
		\eta:[0,T] \rightarrow \bigcup_{t \in [0,T]} X(t) \times \{ t \},\\
		t \mapsto (\bar{\eta}(t), t),
	\end{gather*}
	such that $ \Phi_{-(\cdot)} \bar{\eta}(\cdot) \in L^2(0,T;X(0))$.
	We identify $\eta$ with $\bar{\eta}$.
	This space is equipped with norm
	\[ \|\eta\|_{L^2_X} = \left(\int_0^T \|\eta(t)\|_{X(t)}^2\right)^{\frac{1}{2}},\]
	for $\eta \in L^2_X$.
	If $X(t)$ is a family of Hilbert spaces then this norm is induced by the inner product
	$$ (\eta,\zeta)_{L^2_X} = \int_0^T (\eta(t),\zeta(t))_{X(t)},$$
	for $\eta,\zeta \in L^2_X$.
	In this case, as justified in \cite{alphonse2023function,alphonse2015abstract}, we make the identification $(L^2_{X})^* \cong L^2_{X^*}$.
    In particular for $X(t) = H^1(\Gamma(t))$ we shall write $(L^2_{H^1})^* \cong L^2_{H^{-1}}$, for $H^{-1}(\Gamma(t))$ the dual space of $H^1(\Gamma(t))$.\\
	
	One can similarly define $L^p_X$ for $p \in [1, \infty]$, which is equipped with a norm
	$$ \|\eta\|_{L^p_X} := \begin{cases}
		\left( \int_0^T \|\eta(t)\|^p_{X(t)} \right)^{\frac{1}{p}}, & p \in [1,\infty),\\
		\esssup{t \in [0,T]} \|\eta(t)\|_{X(t)}, & p = \infty.
	\end{cases}
	$$
	We refer the reader to \cite{alphonse2023function} for further details.
\end{definition}

We now state a transport theorem for quantities defined on an evolving surface.
Firstly, we use the following notation for bilinear forms,
\begin{align*}
	m_*(t;\hat{\eta}, \zeta) &:= \langle \hat{\eta}, \zeta \rangle_{H^{-1}(\Gamma(t)) \times H^1(\Gamma(t))},\\
	m(t;\eta,\zeta) &:=  \int_{\Gamma(t)} \eta \zeta, \\
	g(t;\eta,\zeta)  &:=  \int_{\Gamma(t)} \eta \zeta \gradg \cdot V(t),\\
	a(t;\eta,\zeta) &:= \int_{\Gamma(t)} \gradg \eta \cdot \gradg \zeta,
\end{align*}
where $\eta, \zeta \in H^1(\Gamma(t))$, $\hat{\eta} \in H^{-1}(\Gamma(t))$.
The argument in $t$ will typically be omitted.
For weakly differentiable functions we have the following result.
\begin{proposition}[{\cite[Lemma 5.2]{dziuk2013finite}}]
	\label{transport2}
	For all $\eta, \zeta \in H^1_{H^{-1}} \cap L^2_{L^2}$ the map $t \mapsto m(\eta(t), \zeta(t))$ is absolutely continuous and such that
	\[ \frac{d}{dt} m(\eta, \zeta) = m_*(\matdev \eta, \zeta) + m_*(\eta, \matdev \zeta) + g(\eta, \zeta).\]
	Moreover, if $\gradg \matdev \eta, \gradg \matdev \zeta \in L^2_{L^2}$ then $t \mapsto a(\eta(t), \zeta(t))$ is absolutely continuous and such that
	$$ \frac{d}{dt} a(\eta, \zeta) = a(\matdev \eta, \zeta) + a(\eta, \matdev \zeta) + b(\eta, \zeta).$$
	Here
	$$ b(\eta, \zeta) = \int_{\Gamma(t)} \mathbf{B}(V) \gradg \eta \cdot \gradg \zeta , $$
	where
	$$ \mathbf{B}(V) = \left( (\gradg \cdot V)\mathrm{id} - (\gradg V + (\gradg V)^T) \right).$$
\end{proposition}

\section{Positive mobility functions}
In this section we consider a positive mobility function, $M(\cdot)$, such that $M \in C^0(\mbb{R})$ and
\[ 0 < M_1 \leq M(\cdot) \leq M_2, \]
for some positive constants $M_1, M_2$.
This will serve as a more regular mobility function which we can use to approximate the degenerate mobilities later (which vanish at $\pm 1$).
We also only consider the case of a regular potential with polynomial growth conditions, as in \cite{caetano2021cahn}.
More specifically we assume that $F_1, F_2$ are such that the following assumptions hold.
\begin{assumption}
\label{potential assumptions}
    \begin{enumerate}
        \item There exists some constant $\beta \in \mbb{R}$ such that $F(r) \geq \beta$ for all $r \in \mathrm{dom}(F)$.
        \item $F_1$ is convex.
        \item $F_2$ is such that there exists some $\alpha >0$ with
        \[ |F'_2(r)| \leq \alpha (|r| + 1). \]
        \item There exists some $q \in [1, \infty)$ and $\alpha > 0$ such that $|F_1'(r)| \leq \alpha |r|^q + \alpha$.
        \item There exists $\alpha > 0$ such that
        \[ |rF_1'(r)| \leq \alpha (F_1(r) + 1). \]
    \end{enumerate}
\end{assumption}
A standard example of such a potential is the quartic double well function given by
\[ F(r) = \frac{(1-r^2)^2}{4} = \frac{1+r^4}{4} - \frac{r^2}{2}. \]

Here we consider weak solutions to be a pair of functions $(u,w) \in L^\infty_{H^1} \cap H^1_{H^{-1}} \times L^2_{H^1}$ such that
\begin{gather}
	m_*(\matdev u, \phi) + g(u, \phi) + \hat{a}(M(u),w,\phi) = 0, \label{weakeqn1}\\
	m(w,\phi) = \varepsilon a(u, \phi) + \frac{1}{\varepsilon} m(F'(u), \phi), \label{weakeqn2}
\end{gather}
for all $\phi \in H^1(\Gamma(t))$, almost all $t \in [0,T]$, and such that $u(0) = u_0$ almost everywhere on $\Gamma_0$.
Here we have introduced a new trilinear form given by
\begin{gather*}
	\hat{a}(t; \phi, \psi, \chi) = \int_{\Gamma(t)} \phi \gradg\psi \cdot \gradg\chi,
\end{gather*}
where $\phi, \psi, \chi$ are sufficiently smooth functions.
We will omit the argument $t$ throughout, as we have above.\\

The main result of this section is the following.
\begin{theorem}
	\label{regular existence}
	There exists a solution pair $(u,w)\in L^\infty_{H^1} \cap H^1_{H^{-1}} \times L^2_{H^1}$ solving \eqref{weakeqn1}, \eqref{weakeqn2} for all $\phi \in H^1(\Gamma(t))$, almost all $t \in [0,T]$, and such that $u(0) = u_0$.
\end{theorem}

\subsection{Existence}
\subsubsection{Galerkin Approximation}
We proceed by Galerkin approximation, choosing a countable basis of functions given by $\Phi_t \psi_i$, where $\psi_i$ are eigenfunctions of the Laplace-Beltrami operator, $-\lapg$, on $\Gamma_0$.
We denote the finite dimensional subspaces spanned by these (pushforwards of) eigenfunctions as
\[ V^K(t) := \text{span}\{ \Phi_t \psi_i \mid i = 1,...,K \} \subset H^1(\Gamma(t)). \]
It is worth noting that by construction, these basis functions have vanishing material time derivative --- that is $\matdev \Phi_t \psi_i  = 0$ for all $i$.
We define a $H^1$-projection onto $V^K(t)$, denoted $P_K \phi$, as the unique solution of
\[ (P_K \phi, \psi^K)_{H^1(\Gamma(t))} = (\phi, \psi^K)_{H^1(\Gamma(t))}, \]
for all $\psi^K \in V^K(t)$.
Now we consider the approximate solutions, $(u^K, w^K)$, solving
\begin{gather}
	m(\matdev u^K, \phi) + g(u^K, \phi) + \hat{a}(M(u^K),w^K,\phi) = 0, \label{galerkin1}\\
	m(w^K,\phi) = \varepsilon a(u^K, \phi) + \frac{1}{\varepsilon} m(F'(u^K), \phi), \label{galerkin2}
\end{gather}
for all $\phi \in V^K(t)$, almost all $t \in [0,T]$, and such that $u(0) = P_K u_0$ almost everywhere on $\Gamma_0$.

\begin{lemma}
	There exists a short time solution $(u^K, w^K)$ solving \eqref{galerkin1}, \eqref{galerkin2} such that $u^K(0) = P_K u_0$.
\end{lemma}
\begin{proof}
The proof of this result is standard, and follows from the usual appeal to standard ODE theory, namely the Peano existence theorem as the nonlinearities are continuous.
We refer the reader to \cite{caetano2021cahn,caetano2023regularization,elliott2024navier} for details on analogous results.
\end{proof}

Next we show uniform energy estimates to establish long time existence for sufficiently large $K$, and let us pass to the limit $K \rightarrow \infty$.
For this we note the following properties of our choice of basis functions (see also the discussion in \cite{caetano2021cahn,elliott2024navier}):
\begin{enumerate}
	\item The basis $(\psi_i)_i$ of $H^1(\Gamma_0)$ is such that $\psi_1$ is constant on $\Gamma_0$.
	This guarantees that $1 \in V^K(t)$ for all $K \in \mbb{N}, t \in [0,T]$.
	\item We define $P_2^K(t): L^2(\Gamma(t)) \rightarrow V^K(t)$ to be the $L^2$ projection defined by
	\[(P^K_2(t) \phi, \psi)_{L^2(\Gamma(t))} = (\phi, \psi)_{L^2(\Gamma(t))},\ \forall \psi \in V^K(t).\]
	Then for $\eta \in H^1(\Gamma_0)$ that
	\[ \|P_2^K(0) \eta \|_{H^1(\Gamma_0)} \leq C \|\eta\|_{H^1(\Gamma_0)},\] 
	and given $\gamma > 0$ there exists $K^* \in \mbb{N}$ such that for $K > K^*$,
	\[\|P_2^K(0)\eta - \eta\|_{L^2(\Gamma_0)} \leq \gamma \|\eta\|_{H^1(\Gamma_0)}. \]
\end{enumerate}
We notice that this second property implies that for $\eta \in H^1(\Gamma(t))$
\begin{align*}
	\|P_2^K(t)\eta - \eta\|_{L^2(\Gamma(t))} \leq C \gamma \|\eta\|_{H^1(\Gamma(t))},
\end{align*}
and likewise it is straightforward to see that the first property and compatibility of the pairs $(H^1(\Gamma(t)), \Phi_t)$, in the sense of \cite{alphonse2023function,alphonse2015abstract}, implies that
\[ \|P_2^K(t) \eta \|_{H^1(\Gamma(t))} \leq C \|\eta\|_{H^1(\Gamma(t))}.\] 
Lastly we note that $1 \in V^K(t)$ for all $K \in \mbb{N}, t \in [0,T]$ (since it is the first eigenfunction of the Laplace-Beltrami operator on a closed surface), and so one may use the constant function $1$ as a test function so that
\[ \int_{\Gamma(t)} u^K(t) = \int_{\Gamma_0} u_0,\]
We refer the reader to \cite{elliott2024navier} for details on these calculations.

\begin{lemma}\label{energyestimate}
	For sufficiently large $K$, the pair $(u^K,w^K)$ exists on $[0,T]$ and is such that
	\begin{align}
		\sup_{t \in [0,T]} \Ech[u^K;t] + \frac{1}{2}\int_0^T M(u^K)\|\gradg w^K \|_{L^2(\Gamma(t))}^2  \leq C,
	\end{align}
	for a constant, $C$, independent of $K$ but depending on $T$.
\end{lemma}
Here we are using the Ginzburg-Landau functional,
\begin{align}
	\Ech[\phi;t] := \int_{\Gamma(t)} \frac{\varepsilon|\gradg \phi|^2}{2} + \frac{1}{\varepsilon}F(\phi).
\end{align}
It is known that for a stationary, Euclidean domain that the Cahn-Hilliard equation (with a constant mobility) is the $H^{-1}$ gradient flow of the Ginzburg-Landau functional, we refer the reader to \cite{bansch2023interfaces} for details, but even in our present setting this functional is still useful in the analysis.
\begin{proof}
	We begin by testing \eqref{galerkin1} with $w^K$ so that
	\begin{align}
		m(\matdev u^K , w^K) + g(u^K,w^K) + \hat{a}(M(u^K),w^K,w^K) =0, \label{energypf1}
	\end{align}
	and \eqref{galerkin2} with $\matdev u^K \in V^K(t)$,
	\[m(\matdev u^K , w^K) = \varepsilon a(u^K,\matdev u^K) + \frac{1}{\varepsilon}m(F'(u^K),\matdev u^K).\]
	Using the transport theorem, Proposition \ref{transport2}, on the latter equality one finds
	\begin{align}
		m(\matdev u^K , w^K) = \frac{d}{dt}\frac{\varepsilon}{2} a(u^K,u^K) - \frac{\varepsilon}{2} b(u^K,u^K) + \frac{d}{dt} \frac{1}{\varepsilon} m(F(u^K),1) - \frac{1}{\varepsilon} g(F(u^K),1), \label{energypf2}
	\end{align}
	hence using \eqref{energypf2} in \eqref{energypf1} we find
	\begin{align}
		\frac{d}{dt} \Ech[u^K;t] + \hat{a}(M(u^K),w^K,w^K) +g(u^K, w^K)= \frac{\varepsilon}{2} b(u^K,u^K) + \frac{1}{\varepsilon} g(F(u^K),1). \label{energypf3}
	\end{align}
	It remains to bound these terms on the right.
	Firstly we observe from the smoothness assumption on $V$, and the lower bound on $F(\cdot)$, that
	\[\frac{\varepsilon}{2} b(u^K,u^K) + \frac{1}{\varepsilon} g(F(u^K),1) \leq C + C \Ech[u^K;t],  \]
	for a constant, $C$, independent of $t, K$.
	The problematic term here is $g(u^K,w^K)$, for which we use the same argument as in \cite{caetano2021cahn,elliott2024navier}.
	By definition of $P_2^K$ one has that
	\[ g(u^K, w^K) = m(w^K, P_2^K(u^K \gradg \cdot V)), \]
	where $ P_2^K(u^K \gradg \cdot V) \in V^K(t)$ so that one can use this in \eqref{weakeqn2} for
	\[ g(u^K, w^K) = \varepsilon a(u^K, P_2^K(u^K \gradg \cdot V)) + \frac{1}{\varepsilon} m(F'(u^K), P_2^K(u^K \gradg \cdot V)). \]
	It is then clear that 
	\begin{align*}
		g(u^K, w^K) \leq C \|u^K\|_{H^1(\Gamma(t))}^2 + \frac{1}{\varepsilon} g(F'(u^K), u^K) + \frac{1}{\varepsilon} |m(F'(u^K), P_2^K(u^K \gradg \cdot V) - u^K \gradg \cdot V)|,
	\end{align*}
	where it follows from the assumptions on $P_2^K$, and the growth assumptions on $F'$ (Assumption \ref{potential assumptions}), that for any $\gamma > 0$ we can choose some sufficiently large $K$ so that
	\begin{align*}
	    |g(u^K, w^K)| &\leq C + C \|\gradg u^K\|_{H^1(\Gamma(t))}^2 + C \int_{\Gamma(t)} F_1(u^K) + C \gamma \|\gradg u^K\|_{L^2(\Gamma(t))}^{q+1}\\
        &\leq C + C \Ech[u^K; t] + C \gamma \Ech[u^K; t]^{\frac{q+1}{2}}. 
	\end{align*}
	
	Hence we find that \eqref{energypf3} becomes
	\begin{align}
		\frac{d}{dt} \Ech[u^K;t] + \hat{a}(M(u^K),w^K,w^K) \leq C + C \Ech[u^K;t] + C \gamma \Ech[u^K; t]^{\frac{q+1}{2}}, \label{energypf4}
	\end{align}
	where $\gamma > 0$ can be taken arbitrarily small by taking $K$ sufficiently large.\\
	
	We wish to use the Bihari-LaSalle inequality \cite{bihari1956generalization} to conclude, but we require that $\Ech[u^K; t] \geq 0$ for this -- which is not the case.
	However, we know that $\Ech[u^K; t]$ is bounded below, and so we can add some constant, $\tilde{C}$, (independent of $t$) so that
	\[ \widetilde{\Ech}[u^K; t] := \Ech[u^K;t] + \tilde{C} \geq 0, \]
	and consider \eqref{energypf4} with $\Ech$ replaced by $\widetilde{\Ech}$.
	The result then follows as in the proof of \cite[Lemma 5.4]{elliott2024navier}.
	We note that this gives an upper bound in terms of $\widetilde{\Ech}[P_K u_0]$, which we now bound independently of $K$.
	Firstly we see that
	\[ {\Ech}[P_K u_0] \leq C + \frac{\varepsilon}{2} \| \gradg P_K u_0 \|_{L^2(\Gamma_0)}^2 + \frac{1}{\varepsilon}\|F_1(P_K u_0)\|_{L^1(\Gamma_0)} + \frac{1}{\varepsilon}\|F_2(P_K u_0))\|_{L^1(\Gamma_0)}, \]
	and by using the fact that $\|P_K u_0\|_{H^1(\Gamma_0)} \leq \|u_0\|_{H^1(\Gamma_0)}$ as well as the growth conditions for $F_1, F_2$ we see that
	\[ {\Ech}[P_K u_0] \leq C + C \|u_0\|_{H^1(\Gamma_0)}^2 + \|u_0\|_{H^1(\Gamma_0)}^{q+1}, \]
	and so our bound is independent of $K$.
\end{proof}
\subsubsection{Passage to the limit}
From this we have obtained uniform bounds for $u^K$ in $L^\infty_{H^1}$ and $w^K$ in $L^2_{H^1}$, and hence there exist limits $u \in L^\infty_{H^1}$ and $w \in L^2_{H^1}$ such that, up to a subsequence, we have
\begin{gather*}
	u^K \overset{*}{\rightharpoonup} u, \text{ weak-}* \text{ in } L^\infty_{H^1},\\
	w^K {\rightharpoonup} w, \text{ weakly in } L^2_{H^1}.
\end{gather*}
We cannot pass to the limit immediately however, since we do not have any uniform bounds on the time derivative $\matdev u^K$.
Hence to pass to the limit we firstly want to rewrite the equation with an alternate characterisation of the weak material time derivative.
Before we do this however, we observe that arguing along the lines of \cite{caetano2021cahn} Lemma 4.8 one can show that one has strong convergence $u^K \rightarrow u$ in $L^2_{L^2}$, which will be useful in passing to the limit for the nonlinear terms.\\

We recall the following generalisations of the dominated convergence theorem.
\begin{theorem}[{\cite[Theorem B.2]{caetano2021cahn}}]
	\label{generalised dct1}
	Let $(g_n)_{n \in \mbb{N}}$ be a uniformly bounded sequence in $L^{p}_{L^{r}}$, where $p,r \in [1,\infty]$.
	If there exists some $g \in L^p_{L^r}$ such that $g_n \rightarrow g$ almost everywhere then $g_n \rightharpoonup g$ weakly\footnote{In the case where $p = \infty$ this is weak-$*$ convergence.} in $L^p_{L^r}$.
\end{theorem}

\begin{theorem}[{\cite[Theorem 1.20]{evans2015measure}}]
	\label{generalised dct2}
	Let $(g_n)_{n \in \mbb{N}}$ be a sequence of $\mathcal{H}^2$-measurable functions on $\Gamma(t)$ (for a fixed $t$), and $g$ be some $\mathcal{H}^2$-measurable function on $\Gamma(t)$ such that
	\[ g_n \rightarrow g \text{ a.e. on } \Gamma(t).\]
	If there are functions $(h_n)_{n \in \mbb{N}} \subset L^r(\Gamma(t)), h \in L^r(\Gamma(t))$ such that $h_n \rightarrow h$ strongly in $L^r(\Gamma(t))$ and
	\[ |g_n|^p \leq |h_n|^r \text{ a.e. on } \Gamma(t),\]
	then $g_n, g \in L^p(\Gamma(t))$ and $g_n \rightarrow g$ strongly in $L^p(\Gamma(t))$.
\end{theorem}

Now recall that one can obtain a subsequence of $u^K$ such that $u^K \rightarrow u$ pointwise almost everywhere on $\Gamma(t)$ for almost all $t \in [0,T]$.
Thus we use the above results to observe that as $M(u^K) \rightarrow M(u)$ and $F'(u^K) \rightarrow F'(u)$ pointwise almost everywhere, and $M(u) \in L^\infty_{L^\infty}, F'(u) \in L^2_{L^2}$, we see that
\begin{gather*}
	M(u^K) \overset{*}{\rightharpoonup} M(u), \text{ weak-}* \text{ in } L^\infty_{L^\infty},\\
	F'(u^K) {\rightharpoonup} F'(u), \text{ weakly in } L^2_{L^2},
\end{gather*}
where we pass to a subsequence if necessary.
Likewise, from the pointwise almost everywhere convergence of $u^K \rightarrow u$ on $\bigcup_{t \in [0,T]} \Gamma(t) \times \{t  \}$ one can also show strong convergence of the mobility
\[M(u^K) \rightarrow M(u), \text{ strongly in } L^2_{L^2},\]
which will be useful in the argument below.\\

We now use this to pass to the limit $K \rightarrow \infty$ of \eqref{galerkin1}.
Firstly we pick $\zeta \in C^\infty([0,T])$, and for $1 \leq i \leq K$ consider a test function $\zeta \Phi_t \psi_i \in H^1_{V^K}$ in \eqref{galerkin1} and integrate over $[0,T]$ so that
\begin{multline*}
	m(u^K(T), \zeta(T) \Phi_T \psi_i) + m(P_Ku_0,\zeta(0) \psi_i) + \int_{0}^T \zeta(t) \hat{a}(M(u^K),w^K, \Phi_t \psi_i)\\
	= \int_0^T \zeta'(t) m(u^K, \Phi_t \psi_i).
\end{multline*}
If we choose $\zeta \in C_c^\infty([0,T])$ then the above simplifies to
\[\int_{0}^T \zeta(t) \hat{a}(M(u^K),w^K, \Phi_t \psi_i) = \int_0^T \zeta'(t) m(u^K, \Phi_t \psi_i). \]
We now pass to the limit in the obvious way for the term on the right, but for the left term was express this as
\begin{multline*}
	\int_{0}^T \zeta(t) \hat{a}(M(u^K),w^K, \Phi_t \psi_i)  = \underbrace{\int_{0}^T \zeta(t)\hat{a}(M(u^K)- M(u),w^K, \Phi_t \psi_i)}_{\rightarrow 0 \text{ as } K \rightarrow \infty}\\
	+ \underbrace{\int_{0}^T \zeta(t) \hat{a}(M(u),w^K - w, \Phi_t \psi_i)}_{\rightarrow 0 \text{ as } K \rightarrow \infty}
	+\int_{0}^T \zeta(t) \hat{a}(M(u),w, \Phi_t \psi_i),
\end{multline*}
where the first and second terms on the right vanish as $M(u^K) \rightarrow M(u)$ in $L^2_{L^2}$, and $w^K \rightharpoonup w$ in $L^2_{H^1}$ respectively.
Hence we see that
\[\int_{0}^T \zeta(t) \hat{a}(M(u),w, \Phi_t \psi_i) = \int_0^T \zeta'(t) m(u, \Phi_t \psi_i), \]
and so by summing over suitably weighted sums of the above  form one obtains
\[ \int_{0}^T \zeta(t) \hat{a}(M(u),w, \Phi_t P_K\phi) = \int_0^T \zeta'(t) m(u, \Phi_t P_K\phi), \]
for any $\phi \in H^1(\Gamma_0)$.
Now noting that $\Phi_t P_K \phi \rightarrow \Phi_t \phi$ strongly in $L^2_{H^1}$, we pass to the limit $K \rightarrow \infty$ above for
\[ \int_{0}^T \zeta(t) \hat{a}(M(u),w, \Phi_t \phi) = \int_0^T \zeta'(t) m(u, \Phi_t \phi).\]
Hence it follows that $t \mapsto m(u, \Phi_t \phi)$ is weakly differentiable with
\[ \frac{d}{dt}  m(u, \Phi_t \phi) = -\hat{a}(M(u),w, \Phi_t \phi), \]
by definition of the weak derivative.
Clearly one finds that $\hat{a}(M(u),w,\cdot) \in L^2_{H^{-1}}$, and consequently finds that $\matdev u \in L^2_{H^{-1}}$ exists due to \cite[Lemma 2.5]{caetano2021cahn} (see also \cite[Lemma 3.5]{alphonse2015abstract}).
Hence by using Proposition \ref{transport2}, and the fact that $\matdev \Phi_t \phi = 0$, one finds that
\[ m_*(\matdev u, \Phi_t\phi) + g(u, \Phi_t\phi) + \hat{a}(M(u),w,\Phi_t\phi) = 0,\]
for all $\phi \in H^1(\Gamma_0)$ and almost all $t \in [0,T]$.
Noticing that our basis of $H^1(\Gamma(t))$ is formed precisely of functions of the form $\Phi_t\phi$ it is then straightforward to use the above equality to verify that \eqref{weakeqn1} holds.
Passing to the limit in \eqref{galerkin2} is similar but more straightforward as there are no time derivatives.
With these considerations one concludes that Theorem \ref{regular existence} holds.\\

As a final remark to close this subsection, we note that  standard elliptic regularity theory, see \cite{aubin2012nonlinear}, \eqref{weakeqn2} and the regularity of $w$ implies $u \in L^2_{H^2}$, and is such that
\[ \int_{0}^T \|u\|_{H^2(\Gamma(t))}^2 \leq C \int_0^T \left( \|w\|^2_{L^2(\Gamma(t)}+ \|F'(u)\|_{L^2(\Gamma(t))}^2 \right). \]

\subsection{Weak-strong uniqueness}
Here we prove a weak-strong uniqueness result under the assumption that $M \in C^{0,1}(\mbb{R})$.
The main result of this subsection is an improvement on the uniqueness result of \cite{barrett1999finite}, where the authors show a uniqueness result for a class of sufficiently regular solutions.
It is worth noting that our uniqueness result only holds in two dimensions, due to use of interpolation inequalities and the Brezis-Gallou\"et inequality, whereas the result of \cite{barrett1999finite} holds also in three dimensions.
Our approach only requires one strong solution and hence we acquire weak-strong uniqueness.\\

We begin by defining a suitable weak norm to be used in our uniqueness result.
For $z \in L^2_0(\Gamma(t)) := \{ \phi \in L^2(\Gamma(t)) \mid \int_{\Gamma(t)} \phi = 0 \}$, sufficiently function smooth $\xi : \Gamma(t) \rightarrow \mbb{R}$, and a fixed $t \in [0,T]$ we define\footnote{This can be generalised to $z \in H^{-1}(\Gamma(t))$ such that $\langle z, 1 \rangle = 0$, but we do not do this as our application requires further smoothness.} $\invoperator{\xi}z \in H^1(\Gamma(t))$ to be the unique solution of
\begin{align}
	\hat{a}(M(\xi),\invoperator{\xi}z,\phi) = m(z, \phi), \label{invoperator equation}
\end{align}
for all $\phi \in H^1(\Gamma(t)$, and such that $\int_{\Gamma(t)} \invoperator{\xi} z = 0$.
This clearly exists as the assumption $M(\cdot) \geq M_1 > 0$ implies this is a uniformly elliptic equation.
We now define a norm on $L^2_0(\Gamma(t))$ by
\[ \| z \|_\xi^2 := \hat{a}(M(\xi),\invoperator{\xi}z,\invoperator{\xi}z).\]
It is straightforward to show that
\[ M_1 (1+C_P^2) \|\invoperator{\xi}z\|_{H^1(\Gamma(t))}^2 \leq \|z\|_\xi^2 \leq M_2 \|\invoperator{\xi}z\|_{H^1(\Gamma(t))}^2,\]
where $C_P$ is the constant in the Poincar\'e inequality.
Clearly this operator is related to the inverse Laplacian, $\mathcal{G}$, as in \cite{caetano2021cahn,elliott2024navier,EllSalCHLog}, where for a constant $\lambda \in \mbb{R}$ one has
\[ \invoperator{\lambda} = \frac{1}{M(\lambda)} \mathcal{G}. \]
If one has that $\xi \in C^{0,1}(\Gamma(t))$, then the above PDE has Lipschitz continuous coefficients, and from elliptic regularity theory (for instance using \cite[Theorem 8.8]{gilbarg1977elliptic} and a localisation argument) one finds that $\invoperator{\xi}z \in H^2(\Gamma(t))$.
We show a bound for this in the following result, where we make the stronger assumption that $\xi \in H^2(\Gamma(t)) \hookrightarrow C^{0,1}(\Gamma(t))$.
\begin{lemma}
	Let $\xi \in H^2(\Gamma(t))$ and $z \in L^2_0(\Gamma(t))$.
	Then $\invoperator{\xi}z \in H^2(\Gamma(t))$ is such that
	\begin{align}
		\|\invoperator{\xi}z\|_{H^2(\Gamma(t))} \leq C  \|z\|_{L^2(\Gamma(t))} + C \left( \|\xi\|_{H^{1,4}(\Gamma(t))} + \|\xi\|_{H^{1,4}(\Gamma(t))}^2 \right) \|\invoperator{\xi}z\|_{H^{1}(\Gamma(t))}, \label{elliptic regularity}
	\end{align}
	for a constant $C$ independent of $t$.
\end{lemma}
\begin{proof}
	To begin we recall from \cite{dziuk2013finite} that for a function $\phi \in H^2(\Gamma(t))$ one has that
    \[ \|\phi\|_{H^2(\Gamma(t))} \leq C \left( \|\lapg \phi\|_{L^2(\Gamma(t))} + \|\phi\|_{H^1(\Gamma(t))} \right), \]
	where one can show $C$ is independent of $t$.
	We aim to use this inequality with $\phi = \invoperator{\xi}z$, where the issue arises in bounding $\|\lapg \invoperator{\xi}z\|_{L^2(\Gamma(t))}$.
	To do this we observe that
	\begin{align*}
		M_1\|\lapg \invoperator{\xi}z\|_{L^2(\Gamma(t))}^2 &\leq \int_{\Gamma(t)} M(\xi) |\lapg \invoperator{\xi}z|^2\\
		&= \int_{\Gamma(t)} \gradg \cdot(M(\xi) \gradg \invoperator{\xi}z) \lapg \invoperator{\xi}z - \int_{\Gamma(t)} M'(\xi) \gradg \xi \cdot \gradg \invoperator{\xi}z \lapg \invoperator{\xi}z\\
		&= -\int_{\Gamma(t)} z \lapg \invoperator{\xi}z - \int_{\Gamma(t)} M'(\xi) \gradg \xi \cdot \gradg \invoperator{\xi}z \lapg \invoperator{\xi}z,
	\end{align*}
	where we have used the $H^2$ regularity of $\invoperator{\xi}z$ and the fact that $-\divg (M(\xi) \gradg \invoperator{\xi} z) = z$ almost everywhere on $\Gamma(t)$.
	Young's and H\"older's inequalities now yield
	\begin{align*}
		\|\lapg \invoperator{\xi}z\|_{L^2(\Gamma(t))}^2 &\leq C \left( \|z\|_{L^2(\Gamma(t))}^2 + \|\xi\|_{H^{1,4}(\Gamma(t))}^2 \|\invoperator{\xi}z\|_{H^{1,4}(\Gamma(t))}^2 \right).
	\end{align*}
	We now recall (see \cite{aubin2012nonlinear}) the interpolation inequality
	\begin{align}
		\|\phi\|_{H^{1,4}(\Gamma(t))} \leq C \|\phi\|_{H^1(\Gamma(t))}^{\frac{1}{2}} \|\phi\|_{H^2(\Gamma(t))}^{\frac{1}{2}}, \label{H14 interpolation}
	\end{align}
	where it can be shown (by compatibility of function spaces) that $C$ is independent of $t$ --- we refer the reader to \cite[Lemma 3.4]{olshanskii2022tangential} for a proof of this result.
	Using this interpolation inequality one finds
	\begin{align*}
		\|\lapg \invoperator{\xi}z\|_{L^2(\Gamma(t))}^2 &\leq C \left( \|z\|_{L^2(\Gamma(t))}^2 + \|\xi\|_{H^{1,4}(\Gamma(t))}^2 \|\invoperator{\xi}z\|_{H^{1}(\Gamma(t))} \|\invoperator{\xi}z\|_{H^{2}(\Gamma(t))} \right)\\
		&\leq C  \|z\|_{L^2(\Gamma(t))}^2 + C \left( \|\xi\|_{H^{1,4}(\Gamma(t))}^2 + \|\xi\|_{H^{1,4}(\Gamma(t))}^4 \right) \|\invoperator{\xi}z\|_{H^{1}(\Gamma(t))}^2,
	\end{align*}
    where we have used
    \[\|\invoperator{\xi}z\|_{H^2(\Gamma(t))} \leq C \left( \|\lapg \invoperator{\xi}z\|_{L^2(\Gamma(t))} + \|\invoperator{\xi}z\|_{H^1(\Gamma(t))} \right),\]
    and Young's inequality for the second inequality.
\end{proof}

We also require a further estimate on the material derivative of $\invoperator{\xi} z$, which is the content of the following lemma.

\begin{lemma}
	Let $\xi \in H^{1}_{L^2}$, and $z \in H^1_{H^{-1}} \cap L^2_{H^1 \cap L^2_0}$ be such that $(\matdev \xi) \gradg \invoperator{\xi}z \in L^2_{L^2}$.
	Then $\invoperator{\xi} z \in H^1_{H^1}$.
	That is $\matdev \invoperator{\xi} z$ exists as an element of $L^2_{H^1}$, and moreover we have a bound
	\begin{multline*}
		\int_{0}^T \|\matdev \invoperator{\xi}z\|_{H^1(\Gamma(t))}^2 \leq C \int_0^T \left( \|\matdev z \|_{H^{-1}(\Gamma(t))}^2 + \|z\|_{L^2(\Gamma(t))}^2  + \|\invoperator{\xi}z\|_{H^1(\Gamma(t))}^2 \right)\\
		+ C \int_0^T \|\matdev \xi \gradg \invoperator{\xi} z\|_{L^2(\Gamma(t))}^2.
	\end{multline*}
	where $C$ is independent of $\xi$.
\end{lemma}
\begin{proof}
	Firstly one considers $\phi \in H^1_{H^1}$ so that by formally differentiating \eqref{invoperator equation} in time one obtains
	\begin{multline*}
		\hat{a}(M'(\xi) \matdev \xi, \invoperator{\xi}z, \phi) + \hat{a}(M(\xi), \matdev \invoperator{\xi}z, \phi) + \hat{a}(M(\xi),\invoperator{\xi}z, \matdev \phi) + \hat{b}(M(\xi), \invoperator{\xi}z, \phi)\\
		= m_*(\matdev z, \phi) + m(\matdev \phi, z) + g(z, \phi),
	\end{multline*}
	where $M(\cdot)$ is differentiable almost everywhere due to Rademacher's theorem, and
	\[ \hat{b}(M(\xi), \eta, \zeta) = \int_{\Gamma(t)} M(\xi) \mathbf{B}(V) \gradg \eta \cdot \gradg \zeta, \]
	for
	\[\mathbf{B}(V) = \left( (\gradg \cdot V)\mathrm{id} - (\gradg V + (\gradg V)^T) \right),\]
	which follows similarly to Proposition \ref{transport2}.
	Now noting that 
	\[\hat{a}(M(\xi),\invoperator{\xi}z, \matdev \phi)=m(\matdev \phi, z),\]
	one finds that formally
	\begin{align*}
		\hat{a}(M(\xi), \matdev \invoperator{\xi}z, \phi)
		= m_*(\matdev z, \phi) + g(z, \phi) - \hat{a}(M'(\xi) \matdev \xi, \invoperator{\xi}z, \phi) - \hat{b}(M(\xi), \invoperator{\xi}z, \phi),
	\end{align*}
	for all $\phi \in L^2_{H^1}$ (after using a density argument).
	One uses this equality as a definition of $\matdev \invoperator{\xi} z$, and by noting that
	\begin{multline*}
		m_*(\matdev z, 1) + g(z, 1) - \hat{a}(M'(\xi) \matdev \xi, \invoperator{\xi}z, 1) - \hat{b}(M(\xi), \invoperator{\xi}z, 1)\\= \frac{d}{dt} m(z,1) -\frac{d}{dt} \hat{a}(M(\xi), \invoperator{\xi}z,1) = 0,
	\end{multline*}
	one finds that such a function does indeed exist.
	It is then straightforward to verify that $\matdev\invoperator{\xi}z$ is indeed the weak material derivative of $\invoperator{\xi}z$.
	The bound follows from the PDE defining $\matdev \invoperator{\xi} z$ after noting
	\[ |\hat{a}(M'(\xi) \matdev \xi, \invoperator{\xi}z, \phi)| \leq C \|(\matdev \xi) \gradg \invoperator{\xi} z\|_{L^2(\Gamma(t))} \|\gradg \phi\|_{L^2(\Gamma(t))}. \]
	
\end{proof}

We now recall two preliminary results which will be used in proving our uniqueness result.
The first of these results is the following nonlinear Gr\"onwall type result.
\begin{lemma}[{\cite[Lemma 2.2]{li2016tropical}}]
	\label{gronwalltype}
	Let $m_1, m_2, S$ be non-negative functions on $(0,T)$ such that $m_1, S \in L^1(0,T)$ and $m_2 \in L^2(0,T)$, with $S > 0$ a.e. on $(0,T)$.
	Now suppose $f,g$ are non-negative functions on $(0,T)$, $f$ is absolutely continuous on $[0,T)$, such that\footnote{Here we are using the notation $\log^+(x) := \max(0,\log(x))$.}
	\begin{align}
		f'(t) + g(t) \leq m_1(t)f(t) + m_2(t)\left( f(t)g(t) \log^+\left( \frac{S(t)}{g(t)} \right) \right)^{\frac{1}{2}},
	\end{align}
	holds a.e. on $(0,T)$, and $f(0) = 0$.
	Then $f \equiv 0$ on $[0,T)$.
\end{lemma}
Secondly we shall use the following evolving surface analogue of the Brezis-Gallou\"et-Wainger inequality.

\begin{lemma}[{\cite[Lemma 3.11]{elliott2024navier}}]
	\label{brezisgallouet}
    Let $(\Gamma(t))_{t \in [0,T]}$ be a $C^3$ evolving surface\footnote{In fact this inequality will hold for a $C^2$ evolving surface. This allows one to extend our treatment of positive mobility functions to a $C^2$ evolving surface, but for degenerate mobilities will one require a $C^3$ evolving surface due to higher order terms appearing in the existence proof.}.
	Then for $\phi \in H^2(\Gamma(t))$, one has 
	\begin{align}
		\|\phi\|_{L^\infty(\Gamma(t))} \leq C\|\phi\|_{H^1(\Gamma(t))} \left( 1 + \log\left( 1 + \frac{C\|\phi\|_{H^2(\Gamma(t))}}{\|\phi\|_{H^1(\Gamma(t))}} \right)^{\frac{1}{2}} \right),
	\end{align}
	for constants $C$ independent of $t$.
\end{lemma}

\begin{definition}
	We call a pair of functions $(U,W)$ a strong solution pair if it a weak solution pair of \eqref{weakeqn1}, \eqref{weakeqn2} which additionally has that  $U \in L^2_{H^2}\cap H^1_{L^2}$.
\end{definition}
While we do not assume any further smoothness on $W$, one automatically obtains $W \in L^2_{H^2}$ from standard elliptic regularity theory as $\matdev U \in L^2_{L^2}$.

\begin{theorem}\label{weak strong uniqueness1}
	Let $(U,W)$ be a strong solution pair, and $(u,w)$ be a weak solution pair of \eqref{weakeqn1}, \eqref{weakeqn2} for the same initial data, $u_0$.
	If $F_2'$ is Lipschitz continuous, then $U = u$ almost everywhere on $\mathcal{G}_T$.
\end{theorem}
\begin{proof}
	We present a formal argument, for the sake of clarity, which we then justify after the proof. 
	We begin by defining $\bar{U} = U - u$, and $\bar{W} = W - w$.
	It is then clear that
	\begin{gather}
		m_*(\matdev \bar{U}, \phi) + g(\bar{U}, \phi) + \hat{a}(M(U),W,\phi) - \hat{a}(M(u),w,\phi) = 0,\label{unique1}\\
		m(\bar{W},\phi) = \varepsilon a(\bar{U}, \phi) + 	\frac{1}{\varepsilon} m(F'(U) - F'(u), \phi), \label{unique2}
	\end{gather}
	for all $\phi \in H^1(\Gamma(t))$ and almost all $ t \in [0,T]$.
	We then write
	\[ \hat{a}(M(U),W,\phi) - \hat{a}(M(u),w,\phi) = \hat{a}(M(U), \bar{W}, \phi) + \hat{a}(M(U)-M(u),w, \phi), \]
	and choose $\phi = \invoperator{U}\bar{U}$, which we note is well defined, in \eqref{unique1}.
	This yields
	\begin{align}
		m_*(\matdev \bar{U}, \invoperator{U}\bar{U}) + g(\bar{U}, \invoperator{U}\bar{U}) + m(\bar{W}, \bar{U}) + \hat{a}(M(U) - M(u),w,\invoperator{U}\bar{U}) = 0, \label{unique3}
	\end{align}
	where we have used the definition of $\invoperator{U}\bar{U}$.\\
	
	Now we assume for ease of presentation that $\invoperator{U}\bar{U} \in H^1_{H^1}$ and one can compute
	\[m_*(\matdev \bar{U}, \invoperator{U}\bar{U}) + g(\bar{U}, \invoperator{U}\bar{U}) = \frac{d}{dt} m(\bar{U}, \invoperator{U}\bar{U}) - m(\bar{U}, \matdev \invoperator{U} \bar{U}), \]
	where one finds
	\begin{multline*}
		m(\bar{U}, \matdev \invoperator{U} \bar{U}) = \frac{1}{2} \frac{d}{dt} \hat{a}(M(U),\invoperator{U} \bar{U}, \invoperator{U} \bar{U} ) - \frac{1}{2} \hat{b}(M(U),\invoperator{U} \bar{U},\invoperator{U} \bar{U})\\
		- \frac{1}{2} \hat{a}(M'(U) \matdev U, \invoperator{U}\bar{U},\invoperator{U}\bar{U}),
	\end{multline*}
	and hence
	\begin{multline*}
		m_*(\matdev \bar{U}, \invoperator{U}\bar{U}) + g(\bar{U}, \invoperator{U}\bar{U}) = \frac{1}{2} \frac{d}{dt} \|\bar{U}\|_{U}^2 + \frac{1}{2}\hat{b}(M(U),\invoperator{U} \bar{U},\invoperator{U} \bar{U})\\
		+ \frac{1}{2} \hat{a}(M'(U) \matdev U, \invoperator{U}\bar{U},\invoperator{U}\bar{U}).
	\end{multline*}
	Next we test \eqref{unique2} with $\phi = \bar{U}$ for
	\[m(\bar{W},\bar{U}) = \varepsilon a(\bar{U}, \bar{U}) + 	\frac{1}{\varepsilon} m(F'(U) - F'(u), \bar{U}).\]
	Using the monotonicity of $F_1'$ one finds
	\[m(\bar{W},\bar{U}) \geq \varepsilon a(\bar{U}, \bar{U}) + 	\frac{1}{\varepsilon} m(F_2'(U) - F_2'(u), \bar{U}).\]
	
	Thus combining all of the above in \eqref{unique3} one finds
	\begin{multline}
		\frac{1}{2} \frac{d}{dt} \|\bar{U}\|_{U}^2 + \varepsilon \|\gradg \bar{U}\|_{L^2(\Gamma(t)}^2 \leq \frac{1}{2}|\hat{b}(M(U),\invoperator{U} \bar{U},\invoperator{U} \bar{U})| + \frac{1}{2} |\hat{a}(M'(U) \matdev U, \invoperator{U}\bar{U},\invoperator{U}\bar{U})|\\
		- \hat{a}(M(U) - M(u),w,\invoperator{U}\bar{U})- \frac{1}{\varepsilon} m(F_2'(U) - F_2'(u), \bar{U}). \label{unique4}
	\end{multline}
	It then follows that $|\hat{b}(M(U),\invoperator{U} \bar{U},\invoperator{U} \bar{U})| \leq C \|\bar{U}\|_{U}^2$, and using the Lipschitz continuity of $F_2'$
	\begin{align*}
		\frac{1}{\varepsilon} m(F_2'(U) - F_2'(u), \bar{U}) &\leq C \|\bar{U}\|_{L^2(\Gamma(t)}^2\\
		&= C \hat{a}(M(U), \invoperator{U} \bar{U}, \bar{U}) \leq C \|\bar{U}\|_U^2 + \frac{\varepsilon}{2} \|\gradg \bar{U}\|_{L^2(\Gamma(t))}^2. 
	\end{align*}
	To bound the term involving $\matdev U$ we find that
	\begin{align*}
		|\hat{a}(M'(U) \matdev U, \invoperator{U}\bar{U},\invoperator{U}\bar{U})| &\leq C \|\matdev U\|_{L^2(\Gamma(t))} \|\invoperator{U} \bar{U}\|_{H^{1,4}(\Gamma(t))}^2 \\
		& \leq C \|\matdev U\|_{L^2(\Gamma(t))} \|\invoperator{U} \bar{U}\|_{H^{1}(\Gamma(t))} \|\invoperator{U} \bar{U}\|_{H^{2}(\Gamma(t))}.
	\end{align*}
	where we have used \eqref{H14 interpolation}.
    One then bounds the $\|\invoperator{U} \bar{U}\|_{H^{2}(\Gamma(t))}$ term by recalling \eqref{elliptic regularity} so that
	\begin{align*}
		|\hat{a}(M'(U) \matdev U, \invoperator{U}\bar{U},\invoperator{U}\bar{U})| &\leq  C \|\matdev U\|_{L^2(\Gamma(t))} \|\invoperator{U} \bar{U}\|_{H^{1}(\Gamma(t))}\|\bar{U}\|_{L^2(\Gamma(t))}\\
		&+ C \left( \|U\|_{H^{1,4}(\Gamma(t))} + \|U\|_{H^{1,4}(\Gamma(t))}^2 \right) \|\matdev U\|_{L^2(\Gamma(t))} \|\invoperator{U} \bar{U}\|_{H^{1}(\Gamma(t))}^2
	\end{align*}
	and an application of Young's and Poincar\'e's inequalities then yields
	\begin{align*}
		|\hat{a}(M'(U) \matdev U, \invoperator{U}\bar{U},\invoperator{U}\bar{U})| &\leq \frac{\varepsilon}{4} \|\gradg \bar{U}\|_{L^2(\Gamma(t))}^2 + C \|\matdev U\|_{L^2(\Gamma(t))}^2 \| \bar{U}\|_{U}^2\\
		&+ C \left( \|U\|_{H^{1,4}(\Gamma(t))} + \|U\|_{H^{1,4}(\Gamma(t))}^2 \right)\|\matdev U\|_{L^2(\Gamma(t))} \| \bar{U}\|_{U}^2.
	\end{align*}
	
	The most difficult term to bound is
	\[|\hat{a}(M(U) - M(u),w,\invoperator{U}\bar{U})| \leq C \|\bar{U}\|_{L^\infty(\Gamma(t))} \|\gradg w \|_{L^2(\Gamma(t))} \|\bar{U}\|_{U}, \]
	where we have used the Lipschitz continuity of $M(\cdot)$ and H\"older's inequality.
	We now use Lemma \ref{brezisgallouet} to see that
	\begin{align*}
		\|\bar{U}\|_{L^\infty(\Gamma(t))} &\leq C\|\gradg \bar{U}\|_{L^2(\Gamma(t)} \left( 1 + \log\left( 1 + \frac{C\|\bar{U}\|_{H^2(\Gamma(t))}}{\|\bar{U}\|_{H^1(\Gamma(t))}} \right)^{\frac{1}{2}} \right)\\
		& \leq C\|\gradg \bar{U}\|_{L^2(\Gamma(t)} \left( 1 + \log\left( \frac{C \left(\|\gradg \bar{U}\|_{L^2(\Gamma(t))} + \|\bar{U}\|_{H^2(\Gamma(t))}\right)}{\|\gradg \bar{U}\|_{L^2(\Gamma(t))}} \right)^{\frac{1}{2}} \right).
	\end{align*}
	Using these inequalities, and Young's inequality where needed, in \eqref{unique4} one finds
	\begin{multline*}
		\frac{1}{2} \frac{d}{dt} \|\bar{U}\|_{U}^2 + \frac{\varepsilon}{8} \|\gradg \bar{U}\|_{L^2(\Gamma(t))}^2 \leq K(t) \|\bar{U}\|_{U}^2\\
		+ C\|\gradg \bar{U}\|_{L^2(\Gamma(t)} \|\gradg w\|_{L^2(\Gamma(t)} \|\bar{U}\|_U \log\left( \frac{C \left(\|\gradg \bar{U}\|_{L^2(\Gamma(t))} + \|\bar{U}\|_{H^2(\Gamma(t))}\right)}{\|\gradg \bar{U}\|_{L^2(\Gamma(t))}} \right)^{\frac{1}{2}},
	\end{multline*}
	where
	\begin{align*}
		K(t) &=  C \left( 1 + \| \gradg w\|_{L^2(\Gamma(t)}^2 + \|\matdev U\|_{L^2(\Gamma(t))}^2\right)\\
        &+ C\left( \|U\|_{H^{1,4}(\Gamma(t))}+ \|U\|_{H^{1,4}(\Gamma(t))}^2 \right)\|\matdev U\|_{L^2(\Gamma(t))},
	\end{align*}
	and we conclude by using Lemma \ref{gronwalltype}.
\end{proof}

\begin{remark}
	\begin{enumerate}
		\item It follows that $K \in L^1([0,T])$ from the same argument as used in \cite{elliott2024navier}.
		\item The above argument is only formal as we do not necessarily have that $\invoperator{U}\bar{U} \in H^1_{H^1}$ as we only know that $\matdev U \in L^2_{L^2}$.
		To mitigate this we instead consider $\invoperator{\eta_n}\bar{U}$ for a sequence of sufficiently smooth functions $\eta_n$ which converge to $U$ in an appropriate sense.
	\end{enumerate}
\end{remark}

To rigorously justify the previous result we require the following continuity result regarding the operator $\xi \mapsto \invoperator{\xi}$.
\begin{lemma}
	\begin{enumerate}
		\item Let $(\eta_n)_{n \in \mbb{N}} \subset L^{\infty}(\Gamma(t))$ and $\xi \in L^{\infty}(\Gamma(t))$ be such that $\eta_n \rightarrow \xi$ strongly in $L^{\infty}(\Gamma(t))$ as $n \rightarrow \infty$.
		Then for all $z \in L^2_0(\Gamma(t))$ we have strong convergence $\invoperator{\eta_n} z \rightarrow \invoperator{\xi}z$ in $H^1(\Gamma(t))$ as $n \rightarrow \infty$.
		\item Let $(\eta_n)_{n \in \mbb{N}} \subset L^{2p}(\Gamma(t))$ and $\xi \in L^{2p}(\Gamma(t))$ such that $\eta_n \rightarrow \xi$ strongly in $L^{2p}(\Gamma(t))$ as $n \rightarrow \infty$, for some $p \in (1,\infty)$.
		Then for all $z \in L^2_0(\Gamma(t))\cap H^2(\Gamma(t))$ we have strong convergence $\invoperator{\eta_n} z \rightarrow \invoperator{\xi}z$ in $H^1(\Gamma(t))$ as $n \rightarrow \infty$.
	\end{enumerate}
\end{lemma}
\begin{proof}
	We prove only the second statement, as the first statement is similar but slightly easier as it does not require use of elliptic regularity results.
	By definition one has that
	\[ \hat{a}(M(\xi),\invoperator{\xi}z,\phi) = m(z, \phi) = \hat{a}(M(\eta_n),\invoperator{\eta_n}z,\phi), \]
	for all $\phi \in H^1(\Gamma(t))$ and so one readily finds that
	\[ \hat{a}(M(\eta_n),\invoperator{\eta_n}z - \invoperator{\xi} z,\phi) = \hat{a}(M(\xi) - M(\eta_n),\invoperator{\xi}z,\phi).\]
	Thus taking $\phi = \invoperator{\eta_n}z - \invoperator{\xi} z$, and using H\"older's inequality and the assumptions on $M(\cdot)$ we find
	\begin{align*}
	    \|\invoperator{\eta_n}z - \invoperator{\xi} z\|_{H^1(\Gamma(t))}^2 &\leq C \|(M(\xi) - M(\eta_n)) \gradg\invoperator{\xi} z\|_{{L}^2(\Gamma(t))} \|\invoperator{\eta_n}z - \invoperator{\xi} z\|_{H^1(\Gamma(t))} \\
        &\leq C\|\xi - \eta_n\|_{L^{2p}(\Gamma(t))} \|\gradg \invoperator{\xi} z\|_{L^{2p^*}(\Gamma(t))}\|\invoperator{\eta_n}z - \invoperator{\xi} z\|_{H^1(\Gamma(t))},
	\end{align*}
	for $p^*=\frac{p}{p-1}$. 
	Now as $z \in H^2(\Gamma(t))$ we find $\invoperator{\xi} z \in H^2(\Gamma(t))$ and so we use the Sobolev embedding $H^2(\Gamma(t)) \hookrightarrow H^{1,2p^*}(\Gamma(t)))$ for
	\[ \|\invoperator{\eta_n}z - \invoperator{\xi} z\|_{H^1(\Gamma(t))}\leq C\|\xi - \eta_n\|_{L^{2p}(\Gamma(t))} \| \invoperator{\xi} z\|_{H^2(\Gamma(t))}, \]
	whence the result follows.
\end{proof}

It is then straightforward to show the following corollary.
\begin{corollary}
	\begin{enumerate}
		\item Let $(\eta_n)_{n \in \mbb{N}} \subset L^{\infty}(\Gamma(t))$ and $\xi \in L^\infty(\Gamma(t))$ be such that $\eta_n \rightarrow \xi$ strongly in $L^\infty(\Gamma(t))$ as $n \rightarrow \infty$ for almost all $t \in [0,T]$.
		Then for all $z \in C^0_{L^2_0}$ $\|z\|_{\eta_n} \rightarrow \|z\|_\xi$ for almost all $t \in [0,T]$.
		\item  Let $(\eta_n)_{n \in \mbb{N}} \subset L^r_{L^{2p}}$ and $\xi \in L^r_{L^{2p}}$ such that $\eta_n \rightarrow \xi$ strongly in $L^r_{L^{2p}}$, for some given $r \in [1,\infty]$ and $p \in (1,\infty)$.
		Then for all $z \in L^{r^*}_{H^2}$, where $r^* = \frac{r}{r-1}$ we have
		\[ \begin{cases}
			\int_{0}^T \|z\|_{\eta_n}^r \rightarrow \int_0^T \|z\|_\xi^r, & r \in [1,\infty),\\
			\esssup{t \in [0,T]} \|z\|_{\eta_n} \rightarrow \esssup{t \in [0,T]} \|z\|_\xi, & r = \infty.
		\end{cases}\]
		
	\end{enumerate}   
\end{corollary}

Using these results we now justify the previous proof.
We consider smooth functions $\eta_n$ on $\mathcal{G}_T$ such that
\[ \eta_n \rightarrow U \text{ strongly in } L^2_{H^2} \cap H^1_{L^2} \cap L^\infty_{H^1} \text{ as } n \rightarrow \infty. \]
We observe that
\begin{multline*}
	m_*(\matdev \bar{U}, \invoperator{U}\bar{U}) + g(\bar{U}, \invoperator{U}\bar{U}) = \frac{d}{dt} m(\bar{U}, \invoperator{\eta_n}\bar{U}) - m(\bar{U}, \matdev \invoperator{\eta_n} \bar{U})\\
	+ m_*(\matdev \bar{U}, \invoperator{U}\bar{U} - \invoperator{\eta_n} \bar{U}) + g(\bar{U}, \invoperator{U}\bar{U} - \invoperator{\eta_n} \bar{U}),
\end{multline*}
where we find that $\matdev \invoperator{\eta_n} \bar{U} \in H^1_{H^1}$ for all $n \in \mbb{N}$.
From the above proof we recall that
\begin{multline*}
	m_*(\matdev \bar{U}, \invoperator{\eta_n}\bar{U}) + g(\bar{U}, \invoperator{\eta_n}\bar{U}) = \frac{1}{2} \frac{d}{dt} \|\bar{U}\|_{\eta_n}^2 + \frac{1}{2}\hat{b}(M(\eta_n),\invoperator{\eta_n} \bar{U},\invoperator{\eta_n} \bar{U})\\
	+ \frac{1}{2} \hat{a}(M'(\eta_n) \matdev \eta_n, \invoperator{\eta_n}\bar{U},\invoperator{\eta_n}\bar{U}).
\end{multline*}
and one can mirror the calculations above to find that
\begin{multline*}
	\frac{1}{2} \frac{d}{dt} \|\bar{U}\|_{\eta_n}^2 + \frac{\varepsilon}{8} \|\gradg \bar{U}\|_{L^2(\Gamma(t))}^2 \leq K_1(t)\|\bar{U}\|_{\eta_n}^2 + K_2(t) \|\bar{U}\|_{U}^2\\
	+ C\|\gradg \bar{U}\|_{L^2(\Gamma(t))} \|\gradg w\|_{L^2(\Gamma(t))} \|\bar{U}\|_U \log\left( \frac{C \left(\|\gradg \bar{U}\|_{L^2(\Gamma(t))} + \|\bar{U}\|_{H^2(\Gamma(t))}\right)}{\|\gradg \bar{U}\|_{L^2(\Gamma(t))}} \right)^{\frac{1}{2}},
\end{multline*}
where
\begin{gather*}
	K_1(t) = C\left( \|\eta_n\|_{H^{1,4}(\Gamma(t))} + \|\eta_n\|_{H^{1,4}(\Gamma(t))}^2  + \|\matdev \eta_n\|_{L^2(\Gamma(t))}\right)\|\matdev \eta_n\|_{L^2(\Gamma(t))},\\
	K_2(t) =  C \left( 1 + \| \gradg w\|_{L^2(\Gamma(t))}^2\right).
\end{gather*}
We now integrate the above in time, pass to the limit $n \rightarrow \infty$, and conclude by using an appropriate integral form of Lemma \ref{gronwalltype}, which we have included in the appendix (see Lemma \ref{gronwalltype2}).

\subsubsection{Comparison to other uniqueness results}

We now briefly compare Theorem \ref{weak strong uniqueness1} to existing uniqueness results in the literature --- namely \cite[Theorem 2.1]{barrett1999finite} and \cite[Theorem 1.2(B)]{ConGalGat24}.
Firstly, we compare to \cite[Theorem 2.1]{barrett1999finite}.
The uniqueness result of \cite{barrett1999finite} holds for weak solutions satisfying additional regularity properties:
\begin{gather*}
    \ddt{u} \in L^{\frac{8}{8-d}}(0,T;L^2(\Omega)),\\
    u \in L^{2^d}(0,T;H^2(\Omega)),\\
    w \in L^{\frac{8}{6-d}}(0,T;H^1(\Omega)),
\end{gather*}
where $\Omega \subset \mbb{R}^d$ for $d\leq 3$ is either a convex polyhedron or $\partial \Omega$ is $C^{1,1}$.
Inspecting the proof of this result one finds that this additional regularity is only required on one solution, so that the authors may prove a weak-strong uniqueness result for a class of sufficiently strong solutions.
Firstly, this result holds when one extends to 3 dimensional domains, whereas ours does not.
However, even in the 2 dimensional case this result requires a large amount of regularity, namely that $u \in L^4(0,T;H^2(\Omega))$, in comparison to our assumptions on a strong solution ($u \in L^2_{H^2}$).
Their result however does require less regularity on the time derivative in two dimensions, namely $\ddt{u} \in L^\frac{4}{3}(0,T;L^2(\Omega))$ in contrast to our assumption that $\matdev u \in L^2_{L^2}$.\\

Next we compare to the recent uniqueness result of \cite{ConGalGat24}, wherein the authors show that weak solutions are unique in two dimensions, provided that $\partial \Omega$ is $C^3$, and the mobility $M(\cdot)$ is $C^2$.
A key step in their uniqueness proof is establishing the equality \cite[Equation (3.16)]{ConGalGat24}, which in our notation would read
\begin{align*}
    \int_{\Omega} \invoperator{U} \ddt{\bar{U}} \bar{U} &= \frac{1}{2} \frac{d}{dt} \|\bar{U}\|_{U}^2 + \frac{1}{2} \int_{\Omega} \grad \mathcal{G} \ddt{U} M''(U) \grad U |\grad \invoperator{U} \bar{U}|^2\\
    &+ \int_{\Omega} \grad \mathcal{G} \ddt{U} M'(U)\left( D^2\invoperator{U}\bar{U} \grad \invoperator{U} \bar{U} \right),
\end{align*}
where $\mathcal{G}$ is the inverse (Neumann) Laplacian. 
This is justified by approximating by a sequence of smooth functions, similar to our considerations above --- in particular this makes use of the $C^3$ regularity of $\partial \Omega$ (see \cite[(3.21) and (3.22)]{ConGalGat24}).
One immediately notices where the $C^2$ regularity of $M(\cdot)$ is required.
One may expect that their proof immediately transfers to the setting of a sufficiently evolving surface provided $M(\cdot)$ is sufficiently smooth.
However, a subtle distinction between these two problems arises from the fact that while $ \langle \ddt{u},1 \rangle_{H^{-1}(\Omega),H^1(\Omega)} = 0$ on a stationary domain, one has that $\langle\matdev u,1 \rangle_{H^{-1}(\Gamma(t)),H^1(\Gamma(t))} \neq 0$ for a general evolving domain.
One may see this by testing \eqref{weakeqn1} with $\phi = 1$ so that
\[ \langle\matdev u,1 \rangle_{H^{-1}(\Gamma(t)),H^1(\Gamma(t))} = -\int_{\Gamma(t)} u \gradg \cdot V.  \]
Consequently one cannot define $\mathcal{G} \matdev u$ without some further consideration, which we leave for another work.
Two questions naturally arise from the comparison of our uniqueness proof:
\begin{enumerate}
    \item Under sufficient smoothness assumptions can one extend the uniqueness result of \cite{ConGalGat24} to show the uniqueness of weak solutions to \eqref{weakeqn1}, \eqref{weakeqn2}?
    \item For $\Omega$ (or respectively $\Gamma(t)$) a $C^2$ domain (or a $C^2$ evolving surface respectively), and $M(\cdot)$ a $C^{0,1}$ mobility function, are the weak solutions of \eqref{weakeqn1}, \eqref{weakeqn2} unique or is one limited to a weak strong uniqueness result like Theorem \ref{weak strong uniqueness1}?
\end{enumerate}
These are interesting issues to be resolved, among many others such as (weak-strong) uniqueness in higher dimensions (which to the authors knowledge only \cite[Theorem 2.1]{barrett1999finite} has addressed).

\section{Degenerate mobility functions}
We now use the results from the previous section to establish the existence of a weak solution to \eqref{cheqn} for a degenerate mobility function $M(\cdot)$.
We proceed, as in \cite{elliott1991generalized,elliott1996cahn}, by considering a regularised problem and passing to the limit in the regularisation variable.
\subsubsection*{Assumptions}
In this section we consider a degenerate mobility function, $M$, in the sense that $M(r)$ is positive when $r \in (-1,1)$ and $M(\pm 1) = 0$.
A standard example of this is the quadratic mobility, $M(r) = 1 -r^2$.
For our analysis we make the following assumption on the evolution of $\Gamma(t)$.

\begin{assumption}
\label{velocity assumption}
	We assume that the velocity, $V$, of $\Gamma(t)$ is such that $\gradg \cdot V \geq 0$ for all $(x,t) \in \mathcal{G}_T$.
\end{assumption}

This condition also appears in a recent preprint by the authors, concerning evolving surface finite element methods for the Cahn-Hilliard equation with a logarithmic potential \cite{EllSalCHLog}. 
One can equivalently express this condition in more geometric terms, as we see in the following lemma.
\begin{lemma}
	$\gradg \cdot V \geq 0$ if, and only if, for every measurable region $\Sigma(t) \subset \Gamma(t)$,
	\[ \frac{d}{dt} |\Sigma(t)| \geq 0. \]
\end{lemma}
\begin{proof}
	Firstly we calculate that
	\[  \frac{d}{dt} |\Sigma(t)| = \frac{d}{dt} \int_{\Sigma(t)} 1 = \int_{\Sigma(t)} \gradg \cdot V. \]
	Hence, if $\gradg \cdot V \geq 0$ then clearly $\frac{d}{dt} |\Sigma(t)| \geq 0.$
	Conversely, if we assume $\frac{d}{dt} |\Sigma(t)| \geq 0$ for all regions $\Sigma(t)$, then we see
	\[ \int_{\Sigma(t)} \gradg \cdot V \geq 0,\]
	for all regions $\Sigma(t)$.
	As this holds for any $\Sigma(t)$, and $V$ is sufficiently smooth, the usual localisation arguments imply that $\gradg \cdot V \geq 0$ everywhere on $\Gamma(t)$.
\end{proof}
If we assume that the evolution of $\Gamma(t)$ is purely in the normal direction, that is to say $V = V_N \boldsymbol{\nu}$, then $\gradg \cdot V = HV_N$.
Hence an example of suitable evolution is if $V_N = f(H)$, as a function of the mean curvature, where $f(\cdot)$ is sufficiently smooth, and such that $xf(x) \geq 0$.
An explicit example of such an evolution is the inverse mean curvature flow, $f(H) = \frac{1}{H}$.
We refer the reader to \cite{bethuel1999geometric} for a discussion of geometric evolution equations of form $V_N = f(H)$.\\

This is not strictly necessary, as we shall see later, but it does allow one to obtain global in time solutions for a wider class of mobility functions.
Moreover, without this assumption one would expect that the existence result is only local in time, or global in time provided some further assumptions on the initial data.
We refer the reader to \cite[Proposition 5.1]{caetano2021cahn} for details.\\

We also make the following assumption which relates the convex part of the potential, $F_1$, to the mobility function, $M(\cdot)$.
\begin{assumption}
	We assume the mobility and potential functions are such that $M(r)F_1''(r) \in C^0([-1,1])$, and $r F_1'(r) \geq 0$, where $F_1$ is the convex part of the potential as before.
\end{assumption}
In particular we allow $F_1$ to be singular at $\pm 1$, provided the singularities are of a lower order than $M$.
An explicit example of a singular potential and an appropriate mobility function such that this assumption holds is the logarithmic potential,
\begin{align}
	F_1(r) = \frac{\theta}{2}\left( (1+r)\log(1+r)  + (1-r)\log(1-r)\right), \label{log potential}
\end{align}
for some fixed $\theta \in (0,1)$, and a mobility function
\[M(r) = (1-r^2)^k,\]
for $k \geq 1$ .
Moreover, as $F_1(\cdot)$ is convex one can often satisfy the condition $rF_1'(r) \geq 0$ by perturbing $F_1$ with  a quadratic function, and suitably redefining $F_1, F_2$.

\subsubsection*{Weak formulation}
Here we understand a weak solution to \eqref{cheqn} to be a function $u \in L^2_{H^{-1}} \cap L^2_{H^2} \cap H^1_{H^{-1}}$ such that
\begin{align}
	m_*(\matdev u, \phi) + g(u, \phi) - \varepsilon \hat{c}(u,M(u),\phi) - \frac{1}{\varepsilon}\hat{a}(M(u)F''(u),u,\phi) = 0, \label{weakeqn3}
\end{align}
for all $\phi \in H^2(\Gamma(t))$, almost all $t \in [0,T]$, and such that $u(0) = u_0$ almost everywhere on $\Gamma_0$.
Here the trilinear form $\hat{c}(\cdot,\cdot,\cdot)$ is given by
\[ \hat{c}(t;\phi,\psi,\chi) = \int_{\Gamma(t)} \lapg \phi \gradg \cdot (\psi \gradg \chi), \]
where we suppress the dependence on $t$ as usual.\\

This weak formulation follows from straightforward integration by parts on \eqref{cheqn}.
Unlike the weak formulation for a positive mobility, this has not been split into a system of two PDEs, and subsequently requires a more regular space of test functions.
Our main result of this section is the following existence result.

\begin{theorem}
	\label{thm: degenerate existence}
	Let Assumptions \ref{velocity assumption} and \ref{potential assumptions} hold, and $u_0 \in H^1(\Gamma_0)$ be such that
	\begin{align}
		\int_{\Gamma_0} F(u_0) + \Psi(u_0) \leq C < \infty, \label{degenerate initial data}
	\end{align}
	where $\Psi(\cdot)$ is given by \eqref{defn: psi}.
	Then there exists at least one function $u \in H^1_{H^{-1}} \cap L^\infty_{H^1} \cap L^2_{H^2}$ solving \eqref{weakeqn3} for all $\phi \in H^2(\Gamma(t))$, for almost all $t \in [0,T]$, and such that $u(0) = u_0$ almost everywhere on $\Gamma_0$.
\end{theorem}

\subsection{Existence}
\subsubsection{Regularisation}
We consider a regularised problem with a positive mobility as in the previous section.
Throughout we set $\delta \in (0,1)$ as some regularisation parameter, where the limit $\delta \rightarrow 0$ should recover degenerate problem.
We consider the regularised potential given by $F^\delta(r) := F_1^\delta(r) + F_2(r)$, where
\begin{align}
	F_1^\delta(r) :=
	\begin{cases}
		\sum_{k=0}^2 \frac{F_1^{(k)}(-1+\delta)}{k!}(r+1-\delta)^k, & r \leq -1 + \delta \\
		F_1(r), & r \in (-1+\delta, 1- \delta),\\
		\sum_{k=0}^2 \frac{F_1^{(k)}(1-\delta)}{k!}(r+1-\delta)^k, & r \geq 1-\delta,
	\end{cases}
\end{align}
and the regularised mobility given by
\begin{align}
	M^\delta(r) =\begin{cases}
		M(-1+\delta), & r \leq -1 + \delta \\
		M(r), & r \in (-1+\delta, 1- \delta),\\
		M(1-\delta), & r \geq 1-\delta.
	\end{cases}
\end{align}
With this regularised mobility we define a function which will be used when we consider the limit $\delta \rightarrow 0$ to establish that $|u(t)| \leq 1$ almost everywhere.
For $\delta \in (0,1)$ we define $\Psi^\delta: \mbb{R} \rightarrow \mbb{R}^+$ to be the unique solution of the ODE,
\begin{align}
(\Psi^\delta)''(r) = \frac{1}{M^\delta(r)}, \qquad \Psi^\delta(0) = 0 = (\Psi^\delta)'(0). \label{defn: psi delta}
\end{align}
We similarly define $\Psi: (-1,1) \rightarrow \mbb{R}^+$ to be the unique solution of the ODE
\begin{align}
\Psi''(r) = \frac{1}{M(r)}, \qquad \Psi(0) = 0 = \Psi'(0). \label{defn: psi}
\end{align}
We observe that $\Psi^\delta(r) = \Psi(r)$ for $r \in (-1+ \delta, 1- \delta)$.\\

From the above constructions it is clear that $M^\delta \in C^0(\mbb{R})$, in fact it is Lipschitz continuous, and $F_1^\delta \in C^2(\mbb{R})$, and so the preceding theory applies for the existence of a regularised solution for some choice of initial data, $u_0 \in H^1(\Gamma_0)$, such that \eqref{degenerate initial data} holds.
We now show that the corresponding solution is appropriately bounded independent of $\delta$.
To do this, we repeat the energy estimates as in the proof of Lemma \ref{energyestimate}.

\begin{lemma}
	For sufficiently small $\delta$ the solution pair $u^\delta, w^\delta$ is such that
	\begin{align}
		\sup_{t \in [0,T]}\Echd[u^\delta;t] + \int_0^T \int_{\Gamma(t)} M^\delta(u^\delta) |\gradg w^\delta|^2\leq C, \label{reg energy estimate}
	\end{align}
	for a constant $C$ independent of $\delta$ but depending on $T$.
\end{lemma}
\begin{proof}
	Strictly speaking this proof is done at the level of Galerkin approximation, and then passing to the limit in the Galerkin parameter, but we avoid this for notational simplicity.
	By arguing as in Lemma \ref{energyestimate} it is straightforward to show that
	\begin{align}
		\frac{d}{dt} \Echd[u^\delta;t] + \hat{a}(M^\delta(u^\delta),w^\delta,w^\delta) + g(u^\delta,w^\delta) = \frac{\varepsilon}{2} b(u^\delta,u^\delta) + \frac{1}{\varepsilon} g(F^\delta(u),1). \label{energypf5}
	\end{align}
	The key term which is problematic is $g(u^\delta, w^\delta)$, but from our assumption that $\divg V \geq 0$ we will find that this term is dealt with in a straightforward way.
	As $u^\delta \in L^2_{H^2}$ we know that
	\[ w^\delta = -\varepsilon \lapg u^\delta + \frac{1}{\varepsilon} (F^\delta)'(u^\delta), \]
	almost everywhere on $\Gamma(t)$ for almost all $t \in [0,T]$.
	Hence we use this in $g(u^\delta, w^\delta)$ for
	\[ g(u^\delta, w^\delta) = - \varepsilon \int_{\Gamma(t)} u^\delta \lapg u^\delta \divg V + \frac{1}{\varepsilon}\int_{\Gamma(t)} u^\delta (F^\delta)'(u^\delta) \divg V, \]
	and now by integration by parts and noting that $r (F_1^\delta)'(r) \geq 0$ we see
	\begin{multline*}
		g(u^\delta, w^\delta) \geq  \underbrace{\varepsilon \int_{\Gamma(t)} |\gradg u^\delta|^2  \divg V}_{\geq 0} + \varepsilon \int_{\Gamma(t)} u^\delta \gradg u^\delta \cdot \gradg (\divg V)\\
		+ \underbrace{\frac{1}{\varepsilon}\int_{\Gamma(t)} u^\delta (F_1^\delta)'(u^\delta) \divg V}_{\geq 0} + \frac{1}{\varepsilon}\int_{\Gamma(t)} u^\delta F_2'(u^\delta) \divg V.
	\end{multline*}
	Using this in \eqref{energypf5} yields
	\begin{multline}
		\frac{d}{dt} \Echd[u^\delta;t] + \hat{a}(M^\delta(u^\delta),w^\delta,w^\delta)  \leq \frac{\varepsilon}{2} b(u^\delta,u^\delta) + \frac{1}{\varepsilon} g(F^\delta(u),1)\\
		- \varepsilon \int_{\Gamma(t)} u^\delta \gradg u^\delta \cdot \gradg (\divg V) - \frac{1}{\varepsilon}\int_{\Gamma(t)} u^\delta F_2'(u^\delta) \divg V \label{energypf6}.
	\end{multline}
	It is then routine to show that one can appropriately bound this as
	\begin{multline*}
		\frac{\varepsilon}{2} b(u^\delta,u^\delta) + \frac{1}{\varepsilon} g(F^\delta(u),1) + \varepsilon \left|\int_{\Gamma(t)} u^\delta \gradg u^\delta \cdot \gradg (\divg V)\right|+ \frac{1}{\varepsilon}\left|\int_{\Gamma(t)} u^\delta F_2'(u^\delta) \divg V \right|\\
		\leq C + C \Echd[u^\delta],
	\end{multline*}
	after using the Poincaré inequality, bounds on $F_2'$, and the smoothness of $V$.\\
	
	Using bounds of this form in \eqref{energypf6} and integrating in time we see
	\[ \Echd[u^\delta;t] + \int_0^t M^\delta(u^\delta) \gradg w^\delta \cdot \gradg w^\delta \leq \Echd[u_0;0] + C + C \int_0^t \Echd[u^\delta],\]
	and so using the Gr\"onwall inequality yields the desired result.\\
	
	It remains to see that the bound is independent of $\delta$.
	For this we note that
	\begin{multline*}
		\Echd[u_0;0] = \int_{\Gamma(t)} \frac{\varepsilon|\gradg u_0|^2}{2} + \frac{1}{\varepsilon} F^\delta(u_0) = \frac{\varepsilon}{2}\|u_0\|_{H^1(\Gamma_0)}^2 + \frac{1}{\varepsilon} \int_{\Gamma_0}F_2(u_0)\\
		+ \frac{1}{\varepsilon} \int_{\{|u_0|< 1 - \delta\}} F_{1}(u_0) + \frac{1}{\varepsilon} \int_{\{u_0 \geq 1- \delta\}} \sum_{k=0}^2 \frac{F_1^{(k)}(1-\delta)}{k!}(u_0-1+\delta)^k\\
		+ \frac{1}{\varepsilon} \int_{\{u_0 \leq -1 + \delta\}} \sum_{k=0}^2 \frac{F_1^{(k)}(-1+\delta)}{k!}(u_0+1-\delta)^k, 
	\end{multline*}
	and so it is only necessary to bound the last three terms.
	This follows from the monotonicity of $F_1$ and its derivatives near $\pm 1$, for which one can show that for sufficiently small $\delta$
	\begin{multline*}
		\frac{1}{\varepsilon} \int_{\{|u_0|< 1 - \delta\}} F_{1}(u_0)
		+ \frac{1}{\varepsilon} \int_{\{u_0 \geq 1- \delta\}} \sum_{k=0}^2 \frac{F_1^{(k)}(1-\delta)}{k!}(u_0-1+\delta)^k\\
		+ \frac{1}{\varepsilon} \int_{\{u_0 \leq -1 + \delta\}} \sum_{k=0}^2 \frac{F_1^{(k)}(-1+\delta)}{k!}(u_0+1-\delta)^k \leq \frac{1}{\varepsilon} \int_{\Gamma_0} F_1(u_0),
	\end{multline*}
	which is finite by assumption.
	From this we conclude that the desired bound is in fact independent of $\delta$.
\end{proof}
\begin{remark}
    \label{surface divergence remark1}
	This result can be shown to hold without the assumption that $\gradg \cdot V \geq 0$ by using similar arguments to the proof of \cite[Proposition 5.7]{caetano2021cahn}, provided that we have bounds of the form
	\[ |rF_1'(r)| \leq \alpha F_1(r) + \beta, \qquad |F_1'(r)| \leq r F_1'(r) + \alpha, \]
	for some $\alpha > 0$.
    We do not perform this calculation here, but refer the reader to the arguments used in \cite[Proposition 5.7]{caetano2021cahn}.
	It is worth noting also that without the assumption $\gradg \cdot V \geq 0$, one will likely require some further constraints on the initial data.
    For example in \cite{caetano2021cahn, caetano2023regularization} the authors require
	\[ \frac{1}{|\Gamma(t)|} \left|\int_{\Gamma_0} u_0 \right| < 1,\]
	for all $t \in [0,T]$, so that the solution $u$ is valued in $(-1,1)$ almost everywhere (we refer the reader to \cite[Proposition 5.1]{caetano2021cahn}).
	We have used the above argument as it is more succinct, and we will require the assumption $\gradg \cdot V \geq 0$ for the most useful\footnote{See Remark \ref{surface divergence remark2} for a further discussion on removing the assumption $\gradg \cdot V \geq 0$.} form of the next lemma.
\end{remark}
We now show a similar result, which will be used in showing that $|u| \leq 1$ when we pass to the limit $\delta \rightarrow 0$.
\begin{lemma}
	For sufficiently small $\delta$ one has
	\begin{align}
		\sup_{t \in [0,T]} \int_{\Gamma(t)} \Psi^\delta(u^\delta) + \int_0^T \int_{\Gamma(t)} \varepsilon |\lapg u^\delta|^2 + \frac{1}{\varepsilon}  \int_0^T \int_{\Gamma(t)}(F_1^\delta)''(u^\delta)|\gradg u|^2  \leq C, \label{reg energy estimate2}
	\end{align}
	for a constant $C$ independent of $\delta$ but depending on $T$.
\end{lemma}
\begin{proof}
	Firstly, we note that $(\Psi^\delta)'(u^\delta) \in H^1(\Gamma(t))$ for almost all $t \in [0,T]$, and so we may use $(\Psi^\delta)'(u^\delta)$ as an admissible test function in \eqref{weakeqn1}.
	This yields
	\[ m_*(\matdev u^\delta, (\Psi^\delta)'(u^\delta)) + g(u^\delta, (\Psi^\delta)'(u^\delta)) + \hat{a}(M^\delta(u^\delta),w^\delta,(\Psi^\delta)'(u^\delta)) = 0. \]
	We now make some observations about $\Psi^\delta$.
	Firstly, it is clear from the construction that $\Psi^\delta(\cdot)$ is a non-negative function, and it attains its minimum at $0$, with $\Psi^\delta(0) = 0$.
	Next we see that $(\Psi^\delta)'(\cdot)$ is increasing and such that $(\Psi^\delta)'(0) = 0$.
	Hence we have that $r (\Psi^\delta)'(r) \geq 0$ for all $r \in \mbb{R}$.
	In particular, combining this with the assumption that $\divg V \geq 0$ we find that $g(u^\delta, (\Psi^\delta)'(u^\delta)) \geq 0$.
	Now from the transport theorem we know that
	\[ m_*(\matdev u^\delta, (\Psi^\delta)'(u^\delta)) = \frac{d}{dt} m(\Psi^\delta(u^\delta), 1) - g(\Psi^\delta(u^\delta), 1). \]
	Combining these facts we obtain
	\begin{align}
		\frac{d}{dt} m(\Psi^\delta(u^\delta), 1) + \hat{a}(M^\delta(u^\delta), w^\delta, (\Psi^\delta)'(u^\delta)) \leq g(\Psi^\delta(u^\delta),1). \label{energypf7}
	\end{align}
	We now notice that the term involving $\hat{a}$ can be written as
	\[ \hat{a}(M^\delta(u^\delta),w^\delta,(\Psi^\delta)'(u^\delta)) = \int_{\Gamma(t)} M^\delta(u^\delta) \gradg w^\delta \cdot \gradg (\Psi^\delta)'(u^\delta) = \int_{\Gamma(t)} \gradg w^\delta \cdot \gradg u^\delta, \]
	where we have used the fact that $(\Psi^\delta)''(r) = \frac{1}{M^\delta(r)}$.
	Now by integration by parts and using the fact that $w = - \varepsilon \lapg u^\delta + \frac{1}{\varepsilon} (F^\delta)'(u^\delta)$ almost everywhere we see that
	\[ \int_{\Gamma(t)} \gradg w^\delta \cdot \gradg u^\delta = -\int_{\Gamma(t)} w^\delta  \lapg u^\delta = \int_{\Gamma(t)} \left( \varepsilon \lapg u^\delta - \frac{1}{\varepsilon} (F^\delta)'(u^\delta) \right) \lapg u^\delta.\]
	We then use integration by parts on the potential term so that
	\begin{multline*}
		\int_{\Gamma(t)} \left( \varepsilon \lapg u^\delta - \frac{1}{\varepsilon} (F^\delta)'(u^\delta) \right) \lapg u^\delta = \int_{\Gamma(t)} \varepsilon |\lapg u^\delta|^2 + \frac{1}{\varepsilon}\int_{\Gamma(t)} (F_1^\delta)''(u^\delta)|\gradg u|^2\\
		+ \frac{1}{\varepsilon}\int_{\Gamma(t)} F_2''(u^\delta)|\gradg u|^2.
	\end{multline*}
	We combine this with \eqref{energypf7} for
	\begin{multline}
		\frac{d}{dt} m(\Psi^\delta(u^\delta), 1) + \int_{\Gamma(t)} \varepsilon |\lapg u^\delta|^2 + \frac{1}{\varepsilon} (F_1^\delta)''(u^\delta)|\gradg u|^2 \leq g(\Psi^\delta(u^\delta), 1)\\
		- \frac{1}{\varepsilon} \int_{\Gamma(t)} F_2''(u^\delta)|\gradg u|^2. \label{energypf8}
	\end{multline}
	From our assumptions on $V$ we see that
	\[ g(\Psi^\delta(u^\delta),1) \leq C m(\Psi^\delta(u^\delta),1),\]
	for some constant independent of $\delta$, and likewise using the fact that $F_2''$ is uniformly bounded and the estimate \eqref{reg energy estimate} we see
	\[ \left| \frac{1}{\varepsilon} \int_{\Gamma(t)} F_2''(u^\delta)|\gradg u|^2 \right| \leq C, \]
	for a constant independent of $\delta$.
	Hence we find that \eqref{energypf8} becomes
	\[ \frac{d}{dt} m(\Psi^\delta(u^\delta), 1) + \int_{\Gamma(t)} \varepsilon |\lapg u^\delta|^2 + \frac{1}{\varepsilon} (F_1^\delta)''(u^\delta)|\gradg u|^2 \leq C + Cm(\Psi^\delta(u^\delta), 1),\]
	and hence a Gr\"onwall inequality yields \eqref{reg energy estimate2}.
	The independence of $\delta$ follows from the fact that $M(r) \leq M^\delta(r)$, and hence $\Psi^\delta(r) \leq \Psi(r)$ for sufficiently small  $\delta$.
\end{proof}

\begin{remark}
\label{surface divergence remark2}
	We can avoid the assumption that $\gradg \cdot V \geq 0$, provided we make the same assumptions on $F_1$ as we did in Remark \ref{surface divergence remark1}, and we also have a bound of the form
	\[ r (\Psi^\delta)'(r) \leq \alpha \Psi^\delta(r) + \beta, \]
	for some $\alpha > 0, \beta \in \mbb{R}$.
	In this case one finds that
	\[|g(u^\delta, (\Psi^\delta)'(u^\delta))| \leq Cm(\Psi^\delta(u^\delta),1) + C,\]
	and we can proceed as above.
	This is the case for $M(r) = 1 -r^2$, where one finds
	\[ r (\Psi^\delta)'(r) \leq \Psi^\delta(r) + 1. \]
	
	In general this is an unrealistic assumption on a degenerate mobility, as we now illustrate.
	Consider the function $M(r) = (1-r^2)^2$, so that
	\begin{align*}
		r \Psi'(r) &= \int_0^r \frac{d}{ds} s\Psi'(s) \, ds = \int_0^r \Psi'(s) + \frac{s}{(1-s^2)^2} \, ds\\
		&= \Psi(r) + \frac{1}{2(1-r^2)} - \frac{1}{2},
	\end{align*}
	where clearly $\frac{1}{2(1-r^2)}$ is unbounded for $r \in [-1,1]$.
	The calculation for $(\Psi^\delta)'(r)$ follows similarly.
\end{remark}

With this estimate at hand we can now discuss the measure of the set where $|u^\delta| > 1$.
The following lemma is proven similarly to \cite{caetano2021cahn} Lemma 5.8 and \cite{elliott1996cahn} Lemma 2(c).

\begin{lemma}
	For sufficiently small $\delta$ one has that
	\begin{align}
		\int_{\Gamma(t)} [|u^\delta| - 1]_+^2 \leq C \left(\delta^2 + M^\delta(1-\delta) + M^\delta(-1+\delta) \right), \label{delta control}
	\end{align}
	for almost all $t$, where $C$ is independent of $\delta, t$.
	Here we are using the notation $[z]_+ = \max(z,0)$.
\end{lemma}
\begin{proof}
	To begin we write
	\[ \int_{\Gamma(t)} [|u^\delta| - 1]_+^2 = \int_{\{ u^\delta > 1 -\delta \}} (u^\delta -1)^2+ \int_{\{ u^\delta < -1+ \delta \}}  (1-u^\delta)^2. \]
	We focus on the case $u^\delta > 1 - \delta$, and note that the case $u^\delta < -1 + \delta$ follows analogously.
	By construction, we see that for $r > 1 - \delta$ that
	\[ \Psi^\delta(r) = \Psi(1-\delta) + \Psi'(1-\delta)(r - 1 +\delta) + \frac{\Psi''(1-\delta)}{2}(r - 1 +\delta)^2, \]
	and as $\Psi(1-\delta) + \Psi'(1-\delta)(r - 1 +\delta) \geq 0$ for $r \geq 1 - \delta$ we see
	\[ (r - 1 +\delta)^2 \leq 2 M(1- \delta) \Psi^\delta(r),\]
	where we have used the fact that $\Psi''(r) = \frac{1}{M(r)}$.
	Then for $r > 1- \delta$ we see that
	\[ (r-1)^2 \leq \delta^2 + (r-1+\delta)^2. \]
	From this we find that
	\[ \int_{\{ u^\delta > 1 -\delta \}} (u^\delta -1)^2 \leq \int_{\Gamma(t)} \left( \delta^2 + 2 M(1- \delta) \Psi^\delta(u^\delta) \right), \]
	and one can similarly show that
	\[ \int_{\{ u^\delta < -1 +\delta \}} (u^\delta -1)^2 \leq \int_{\Gamma(t)} \left( \delta^2 + 2 M(-1+ \delta) \Psi^\delta(u^\delta) \right). \]
	Combining these with \eqref{reg energy estimate2} one finds that
	\[ \int_{\{ u^\delta > 1 -\delta \}} (u^\delta -1)^2+ \int_{\{ u^\delta < -1+ \delta \}}  (1-u^\delta)^2 \leq C \left(\delta^2 + M^\delta(1-\delta) + M^\delta(-1+\delta) \right), \]
	as desired, where we have used a uniform bound on $|\Gamma(t)|$.
\end{proof}
This will be used, in the limit $\delta \rightarrow 0$, to show that $|u|\leq 1$, where $u$ is the limit of $u^\delta$ in some weak sense (see \S \ref{subsubsection: degmob limit}).\\

Lastly we show a uniform bound on the time derivative.
\begin{lemma}
	For sufficiently small $\delta$ we have that
	\begin{align}
		\int_0^T \int_{\Gamma(t)}M^\delta(u^\delta)^2 |\gradg w^\delta|^2 \leq C, \label{reg energy estimate3}\\
		\int_0^T \|\matdev u^\delta\|_{H^{-1}(\Gamma(t))}^2 \leq C, \label{reg energy estimate4}
	\end{align}
	for constants, $C$, independent of $\delta$.
\end{lemma}
\begin{proof}
	We firstly establish \eqref{reg energy estimate3}, which is currently the missing ingredient in showing \eqref{reg energy estimate4}.
	To begin one writes
	\begin{multline*}
		\int_{\Gamma(t)}M^\delta(u^\delta)^2 |\gradg w^\delta|^2 = \int_{\{ |u^\delta| \leq 1 - \delta \}}M(u^\delta) M^\delta(u^\delta) |\gradg w^\delta|^2\\
		+ \int_{\{ u^\delta > 1 - \delta \}}M(1-\delta) M^\delta(u^\delta) |\gradg w^\delta|^2
		+ \int_{\{ u^\delta < -1 + \delta \}}M(-1 +\delta) M^\delta(u^\delta) |\gradg w^\delta|^2, 
	\end{multline*}
	from which we immediately obtain
	\[ \int_{\Gamma(t)}M^\delta(u^\delta)^2 |\gradg w^\delta|^2  \leq 3 \max_{r \in [-1,1]} M(r) \int_{\Gamma(t)}M^\delta(u^\delta) |\gradg w^\delta|^2, \]
	and \eqref{reg energy estimate3} then follows after using \eqref{reg energy estimate}.\\
	
	Establishing \eqref{reg energy estimate4} is now straightforward, as by using \eqref{weakeqn1} one finds that
	\[ m_*(\matdev u^\delta, \phi) = -g(u^\delta, \phi) - \hat{a}(M^\delta(u^\delta), w^\delta, \phi),\]
	for all $\phi \in H^1(\Gamma(t))$ and almost all $t \in [0,T]$.
	In particular this means that
	\[ |m_*(\matdev u^\delta, \phi)| \leq C \|u^\delta\|_{L^2(\Gamma(t)}\|\phi\|_{H^1(\Gamma(t))} + \|M^\delta(u^\delta) \gradg w^\delta\|_{L^2(\Gamma(t))} \|\phi\|_{H^1(\Gamma(t))}, \]
	and one concludes that, for $\phi \neq 0$,
	\[ \int_0^T \frac{|m_*(\matdev u^\delta, \phi)|^2}{\|\phi\|_{H^1(\Gamma(t))}^2} \leq C \int_0^T \|u^\delta\|_{L^2(\Gamma(t))}^2 + \int_0^T \int_{\Gamma(t)}M^\delta(u^\delta)^2 |\gradg w^\delta|^2,  \]
	and \eqref{reg energy estimate4} follows from this and \eqref{reg energy estimate}, \eqref{reg energy estimate3}.
\end{proof}
\subsubsection{Passage to the limit}
\label{subsubsection: degmob limit}
We now pass to the limit $\delta \rightarrow 0$.
From our bounds in \eqref{reg energy estimate}, \eqref{reg energy estimate2}, and \eqref{reg energy estimate4} we have uniform bounds on $u^\delta$ in $L^\infty_{H^1}$, $\lapg u^\delta$ in $L^2_{L^2}$, and $\matdev u^\delta$ in $L^2_{H^{-1}}$.
In particular, this lets us obtain weakly convergent subsequences of $u^\delta$ such that
\begin{gather*}
	u^\delta \overset{*}{\rightharpoonup} u, \text{ weak-}* \text{ in } L^\infty_{H^1},\\
	\matdev u^\delta {\rightharpoonup} \matdev u, \text{ weakly in } L^2_{H^{-1}},\\
	\lapg u^\delta {\rightharpoonup} \lapg u, \text{ weakly in } L^2_{L^2}.
\end{gather*}
In particular we see that $u \in L^\infty_{H^1} \cap H^1_{H^{-1}} \cap L^2_{H^2}$, and so from the (uniform in time) compact embedding $H^2(\Gamma(t)) \overset{c}{\hookrightarrow} H^1(\Gamma(t))$ and the continuous embedding $H^1(\Gamma(t)) \hookrightarrow H^{-1}(\Gamma(t))$ one may apply the Aubin-Lions result from \cite[Theorem 5.2]{alphonse2023function}  to obtain a strongly convergent subsequence such that
\begin{gather*}
	u^\delta \rightarrow u, \text{ strongly in } L^2_{H^1},
\end{gather*}
and similarly from $H^1(\Gamma(t)) \overset{c}{\hookrightarrow} L^2(\Gamma(t)) \hookrightarrow H^{-1}(\Gamma(t))$
\begin{gather*}
	u^\delta \rightarrow u, \text{ strongly in } L^p_{L^2},
\end{gather*}
for all $p \in (1, \infty)$.
Owing to \eqref{delta control} we know that this limiting function, $u$, is such that $|u(t)| \leq 1$ almost everywhere on $\Gamma(t)$ for almost all $t \in [0,T]$.
Finally, from \eqref{reg energy estimate3} we see that we have 
\[ \mbf{q}^\delta {\rightharpoonup} \mbf{q}, \text{ weakly in } L^2_{L^2}, \]
where $\mbf{q}^\delta = M^\delta(u^\delta) \gradg w^\delta$, for some $\mbf{q} \in L^2_{L^2}$.
We wish to now identify $\mbf{q}$.\\

We consider an arbitrary $\boldsymbol{\phi} \in L^2_{H^1} \cap L^\infty_{L^\infty}$ and look at $\int_0^T \int_{\Gamma(t)} \mbf{q}^\delta \cdot \boldsymbol{\phi}$.
One can show, by using elliptic regularity and the facts that $w^\delta \in L^2_{H^1}, F^\delta(u^\delta) \in L^2_{H^1}$, where the latter of these uses the boundedness of $(F^\delta)''(\cdot)$, that $u^\delta \in L^2_{H^3}$.
However, this bound is not uniform in $\delta$, but will nonetheless be useful in determining $\mbf{q}$.
With this in mind we may use that $w^\delta = -\varepsilon \lapg u^\delta + \frac{1}{\varepsilon} (F^\delta)'(u^\delta)$ almost everywhere to write
\begin{align*}
	\int_0^T \int_{\Gamma(t)} \mbf{q}^\delta \cdot \boldsymbol{\phi} &= \int_0^T \int_{\Gamma(t)} M^\delta(u^\delta) \gradg w^\delta \cdot \boldsymbol{\phi}\\
	&= \int_0^T \int_{\Gamma(t)} M^\delta(u^\delta) \gradg\left( -\varepsilon \lapg u^\delta + \frac{1}{\varepsilon} (F^\delta)'(u^\delta) \right) \cdot \boldsymbol{\phi}.
\end{align*}
Now by using integration by parts it is straightforward to see that
\begin{multline*}
	\int_0^T \int_{\Gamma(t)} \mbf{q}^\delta \cdot \boldsymbol{\phi} = \underbrace{\int_0^T \int_{\Gamma(t)} \varepsilon M^\delta(u^\delta)\lapg u^\delta \divg \boldsymbol{\phi}}_{=: I_1}\\
	+ \underbrace{\int_0^T \int_{\Gamma(t)} \varepsilon (M^\delta)'(u^\delta)\lapg u^\delta \gradg u^\delta \cdot \boldsymbol{\phi}}_{=: I_2}
	+ \underbrace{\int_0^T \int_{\Gamma(t)} \frac{1}{\varepsilon} M^\delta(u^\delta)(F^\delta)''(u^\delta)\gradg u^\delta \cdot \boldsymbol{\phi}}_{=:I_3}.
\end{multline*}
By definition of $\mbf{q}$ we see the left-hand side clearly converges to $\int_0^T \int_{\Gamma(t)} \mbf{q} \cdot \boldsymbol{\phi}$ as $\delta \rightarrow 0$.
We now consider the limit of $I_1, I_2, I_3$ as $\delta \rightarrow 0$.
For $I_1$ we note that $M^\delta(\cdot) \rightarrow M(\cdot)$ uniformly so that one can show
\[ M^\delta(u^\delta) \rightarrow M(u), \]
pointwise almost everywhere on $\mathcal{G}_T$.
Then as $M^\delta(\cdot)$ is uniformly bounded, and $\lapg u^\delta \rightharpoonup \lapg u$ in $L^2_{L^2}$ one has
\[ \int_0^T \int_{\Gamma(t)} \varepsilon M^\delta(u^\delta)\lapg u^\delta \divg \boldsymbol{\phi} \rightarrow \int_0^T \int_{\Gamma(t)} \varepsilon M(u)\lapg u \divg \boldsymbol{\phi}, \]
as $\delta \rightarrow 0$.
Next we consider $I_2$, for which we claim the following, mirroring the analogous proof in \cite{elliott1996cahn}.\\
{\bf Claim: }$(M^\delta)'(u^\delta) \gradg u^\delta \rightarrow M'(u) \gradg u$ strongly in $L^2_{L^2}$.\\
Assuming this claim holds then one clearly has
\[ \int_0^T \int_{\Gamma(t)} \varepsilon (M^\delta)'(u^\delta)\lapg u^\delta \gradg u^\delta \cdot \boldsymbol{\phi} \rightarrow \int_0^T \int_{\Gamma(t)} \varepsilon M'(u)\lapg u \gradg u \cdot \boldsymbol{\phi}, \]
by using the claimed strong convergence, the weak convergence $\lapg u^\delta \rightharpoonup \lapg u$ in $L^2_{L^2}$, and the fact that $\boldsymbol{\phi} \in L^\infty_{L^\infty}$.
To see that this claim does indeed hold we write
\begin{align*}
	&\int_0^T \int_{\Gamma(t)} \left| (M^\delta)'(u^\delta) \gradg u^\delta - M'(u) \gradg u \right|^2\\
	&= \underbrace{\int_0^T \int_{\{ |u(t)| = 1 \}} \left| (M^\delta)'(u^\delta) \gradg u^\delta - M'(u) \gradg u \right|^2}_{=: J_1}\\
	&+ \underbrace{\int_0^T \int_{\{ |u(t)| < 1\}} \left| (M^\delta)'(u^\delta) \gradg u^\delta - M'(u) \gradg u \right|^2}_{=:J_2}.
\end{align*}
It can be shown that $\gradg u(t) = 0$ on the set $\{|u(t)| = 1 \}$ (see for example {\cite[Lemma 7.7]{gilbarg1977elliptic}} for the proof in a Euclidean setting), and hence we find that
\[ J_1 =\int_0^T \int_{\{ |u(t)| = 1 \}} \left| (M^\delta)'(u^\delta) \gradg u^\delta \right|^2 \leq C \int_0^T \int_{\{ |u(t)| = 1 \}} \left|\gradg u^\delta \right|^2, \]
where we have used the uniform bound for $M^\delta(\cdot)$.
Passing to the limit we have that $\gradg u^\delta \rightarrow \gradg u$ strongly in $L^2_{L^2}$ and so this integral vanishes as above.
Next to see that $J_2$ vanishes we note that $(M^\delta)'(u^\delta) \rightarrow M'(u)$ almost everywhere on the set $\{ u(t) < 1 \}$, for almost all $t \in [0,T]$.
Hence from the strong $L^2_{L^2}$ convergence of $\gradg u^\delta$ one finds $(M^\delta)'(u^\delta) \gradg u^\delta \rightarrow M'(u) \gradg u$ almost everywhere on $\{ u(t) < 1 \}$, for almost all $t \in [0,T]$, and so Theorem \ref{generalised dct2} shows that $J_2$ vanishes as $\delta \rightarrow 0$.\\

It remains to discuss the limit of $I_3$.
Here one uses the definition of $M^\delta, F^\delta$ to split this integral up as
\begin{align*}
	I_3 &= \underbrace{\int_{0}^T\int_{\{ |u^\delta| < 1 -\delta \}} M(u^\delta)F''(u^\delta) \gradg u^\delta \cdot \boldsymbol{\phi}}_{=: J_3}\\
	&+ \underbrace{\int_{0}^T\int_{\{ u^\delta \geq 1 -\delta \}} M(1- \delta)F''(1-\delta) \gradg u^\delta \cdot \boldsymbol{\phi}}_{=: J_4}\\
	&+ \underbrace{\int_{0}^T\int_{\{ u^\delta \leq -1 +\delta \}} M(-1+ \delta)F''(-1+\delta) \gradg u^\delta \cdot \boldsymbol{\phi}}_{=:J_5}.
\end{align*}
For $J_3$ we use Theorem \ref{generalised dct1} after observing that
\[M(u^\delta)F''(u^\delta) \gradg u^\delta \rightarrow M(u)F''(u) \gradg u\]
pointwise almost everywhere on $\Gamma(t)$.
For $J_4$ we note
\begin{align*}
	&\lim_{\delta \rightarrow 0} \int_{0}^T\int_{\{ u^\delta \geq 1 -\delta \}} M(1- \delta)F''(1-\delta) \gradg u^\delta \cdot \boldsymbol{\phi}\\
	&= \left(\lim_{\delta \rightarrow 0} M(1- \delta)F''(1-\delta) \right) \left( \lim_{\delta \rightarrow 0} \int_{0}^T\int_{\{ u^\delta \geq 1 -\delta \}} \gradg u^\delta \cdot \boldsymbol{\phi} \right)\\
	&=\left(\lim_{\delta \rightarrow 0} M(1- \delta)F''(1-\delta) \right) \left(  \int_{0}^T\int_{\{ u = 1 \}} \gradg u \cdot \boldsymbol{\phi}\right)\\
	&= 0,
\end{align*}
and similarly we find that $\lim_{\delta \rightarrow 0} J_5 = 0$.
Thus this establishes the existence of a solution in the sense of \eqref{weakeqn3}.

For a constant mobility it is known that the solution of the Cahn-Hilliard equation with a logarithmic potential does not attain the pure phases $u = \pm 1$ on a set of positive measure, as shown in \cite{caetano2021cahn,caetano2023regularization}.
This is not typically true for a degenerate mobility, but one can show the following result.

\begin{lemma}
	\label{delta control2}
	If $\Psi$ is such that $\lim_{r \rightarrow \pm 1} \Psi(r) = \infty$, then the solution $u$ as constructed above is such that $|u| < 1$ almost everywhere on $\Gamma(t)$ for almost all $t \in [0,T]$.
\end{lemma}
\begin{proof}
	This result follows from the same logic as in \cite{caetano2021cahn,elliott1996cahn}.
	We recall that we have shown for almost all $t \in [0,T]$.
	\[ \int_{\Gamma(t)} \Psi^\delta (u^\delta) \leq C, \]
	for some constant independent of $\delta$.
	From Fatou's lemma one has that
	\[\int_{\Gamma(t)} \liminf_{\delta \rightarrow 0} \Psi^\delta(u^\delta) \leq \liminf_{\delta \rightarrow 0} \int_{\Gamma(t)} \Psi^\delta(u^\delta) \leq C.\]
	We now claim that
	\[\liminf_{\delta \rightarrow 0} \Psi^\delta(u^\delta) = \begin{cases}
		\Psi(u), & |u| < 1,\\
		\infty, & \mathrm{otherwise,}
	\end{cases}\]
	almost everywhere on $\Gamma(t)$ and for almost all $t \in [0,T]$.
	From this claim it follows that since $\int_{\Gamma(t)} \liminf_{\delta \rightarrow 0} \Psi^\delta(u^\delta) < \infty$ one must have that $|u| < 1$ almost everywhere on $\Gamma(t)$ for almost all $t \in [0,T]$.\\
	
	We now prove the above claim.
	Firstly, we recall that we have pointwise convergence of $u^\delta \rightarrow u$ almost everywhere on $\Gamma(t)$ for almost all $t$.
	To begin we assume we have a point $(x,t)$ such that the pointwise convergence holds and $|u(x,t)| < 1$.
	We fix $\gamma > 0$ to be some arbitrarily small constant.
	By construction one can find $\delta_1$ such that $\Psi(u(x,t)) = \Psi^\delta(u(x,t))$ for all $\delta < \delta_1$.
	Similarly by continuity of $\Psi$ on $(-1,1)$ one can find $\delta_2$ such that
	\[ |r - u(x,t)| < \delta_2 \Rightarrow |\Psi(r) - \Psi(u(x,t))| < \gamma. \]
	Lastly by pointwise convergence $u^\delta(x,t) \rightarrow u(x,t)$ for a subsequence of $\delta \rightarrow 0$ one has $\delta_3$ such that
	\[\delta < \delta_3 \Rightarrow |u^\delta(x,t) - u(x,t)| < \delta_2.\]
	Thus taking $\delta < \min\{\delta_1, \delta_2, \delta_3\}$ one finds
	\[ |\Psi^\delta(u^\delta(x,t)) - \Psi(u(x,t))| < \gamma, \]
	where $\gamma$ was arbitrary.\\
	
	Now we suppose that $|u(x,t)| = 1$.
	If $u(x,t) =1$ then one has
	\[ \Psi^\delta(u^\delta) \geq \min\{ \Psi(1-\delta), \Psi(u^\delta(x,t)) \} \rightarrow \infty,  \]
	as $\delta \rightarrow 0$.
	A similar argument holds when $u(x,t) = -1$.
\end{proof}

\subsection{Strong uniqueness away from the pure phases}
Here we use similar arguments to the proof of Theorem \ref{weak strong uniqueness1} to show uniqueness of strong solutions, with some restriction on the initial data.
We also assume that the function $r \mapsto M(r)F''(r)$ is locally Lipschitz continuous.
This assumption is not particular restrictive and holds, for example, for the logarithmic potential \eqref{log potential} and mobility functions $M(r) = (1-r^2)^k$ for $k \geq 1$.
\begin{definition}
	We call a function $U$ a strong solution if it is a weak solution to \eqref{weakeqn3} and one has additionally that $U \in H^1_{L^2}$ and $\gradg \lapg U \in L^2_{L^2}$.
\end{definition}

Here we assume that $|u_0| \leq 1 - \gamma_1$, for some small constant $\gamma_1 > 0$ so that $\invoperator{u_0}$ is a well-defined operator, and that a solution $u$ with initial data $u_0$ is such that $|u| \leq 1 - \gamma_2$, almost everywhere on $\bigcup_{t \in [0,t^*]} \Gamma(t) \times \{ t \}$, for some $t^* > 0$, and a small constant $\gamma_2 >0$.

\begin{theorem}
	\label{weak strong uniqueness2}
	Let $U_1, U_2$ be strong solutions of \eqref{weakeqn3} with the same initial data, $u_0$, such that $|u_0| \leq 1 - \gamma_1$ almost everywhere on $\Gamma_0$ for some constant $\gamma_1 > 0$.
	Then if there exists some time interval $[0,t^*] \subseteq [0,T]$ such that $|U_i| \leq 1 - \gamma_2$ almost everywhere on $\Gamma(t)$ for almost all $t \in [0,t^*]$, for some constant $\gamma_2 > 0$ and $i=1,2$, then we have $U_1 = U_2$ almost everywhere on $\mathcal{G}_{t^*}$.
\end{theorem}
\begin{proof}
	As before we present a formal argument for the sake of clarity, which can be rigorously justified as we did for Theorem \ref{weak strong uniqueness1}.
	For ease of notation define $G(r):= M(r)F''(r)$.
	Define $\bar{U} = U_1 - U_2$, then one finds
	\begin{multline}
		m_*(\matdev \bar{U}, \phi) + g(\bar{U}, \phi) - \varepsilon \hat{c}(U_1,M(U_1),\phi) + \varepsilon \hat{c}(U_2,M(U_2),\phi)\\
		- \frac{1}{\varepsilon}\hat{a}(G(U_1),U_1,\phi) + \frac{1}{\varepsilon}\hat{a}(G(U_2),U_2,\phi) = 0, \label{unique5}
	\end{multline}
	where one then rewrites the final four terms as
	\begin{gather*}
		\hat{c}(U_1,M(U_1),\phi) - \hat{c}(U_2,M(U_2),\phi) = \hat{c}(\bar{U},M(U_1),\phi) + \hat{c}(U_2,M(U_1)-M(U_2),\phi),\\
		\hat{a}(G(U_1),U_1,\phi)- \hat{a}(G(U_2),U_2,\phi) = \hat{a}(G(U_1)- G(U_2),U_1,\phi)+ \hat{a}(G(U_2),\bar{U},\phi).
	\end{gather*}
	Now by our smallness assumption we find that $\invoperator{U_1}\bar{U}$ is well defined and so we may test \eqref{unique5} with $\phi = \invoperator{U_1}\bar{U} \in L^2_{H^2}$ to obtain
	\begin{multline*}
		m_*(\matdev \bar{U}, \invoperator{U_1}\bar{U}) + g(\bar{U}, \invoperator{U_1}\bar{U}) - \varepsilon \hat{c}(\bar{U},M(U_1),\invoperator{U_1}\bar{U})\\
		= \varepsilon \hat{c}(U_2,M(U_1)-M(U_2),\invoperator{U_1}\bar{U})
		- \frac{1}{\varepsilon}\hat{a}(G(U_1)- G(U_2),U_1,\invoperator{U_1}\bar{U})\\
		+ \frac{1}{\varepsilon} \hat{a}(G(U_2),\bar{U},\invoperator{U_1}\bar{U}),
		\end{multline*}
	We recall from the proof of Theorem \ref{weak strong uniqueness1} that if one assumes that $\invoperator{U_1}\bar{U} \in H^1_{H^1}$ then
	\begin{multline}
		m_*(\matdev \bar{U}, \invoperator{U_1}\bar{U}) + g(\bar{U}, \invoperator{U_1}\bar{U}) = \frac{1}{2} \frac{d}{dt} \|\bar{U}\|_{U_1}^2 + \frac{1}{2}\hat{b}(M(U_1),\invoperator{U_1} \bar{U},\invoperator{U_1} \bar{U})\\+ \frac{1}{2} \hat{a}(M'(U_1) \matdev U_1, \invoperator{U_1} \bar{U}, \invoperator{U_1} \bar{U}),\label{unique6}\end{multline}
	where the two rightmost terms are handled in the exact same manner as in the proof of Theorem \ref{weak strong uniqueness1}.
	Next we observe that
	\begin{align}
		\hat{c}(\bar{U},M(U_1),\invoperator{U_1}\bar{U}) &= \int_{\Gamma(t)} \lapg \bar{U} \gradg \cdot (M(U_1) \gradg \invoperator{U_1}\bar{U}) \notag \\&
		= -\int_{\Gamma(t)} \bar{U}\lapg \bar{U}= \int_{\Gamma(t)} |\gradg \bar{U}|^2 \label{unique7}
	\end{align}
	where we have used the fact that $ -\gradg \cdot (M(U_1) \gradg \invoperator{U_1}\bar{U}) = \bar{U}$ almost everywhere on $\Gamma(t)$.\\
	
	By H\"older's inequality and the assumption that $G(\cdot)$ is (locally) Lipschitz continuous one finds that
	\[ |\hat{a}(G(U_1)- G(U_2),U_1,\invoperator{U_1}\bar{U})| \leq C \|\bar{U}\|_{L^\infty(\Gamma(t))} \|\gradg U_1\|_{L^2(\Gamma(t))} \| \bar{U}\|_{U_1}. \]
	By using Lemma \ref{brezisgallouet} to control the $\|\bar{U}\|_{L^\infty(\Gamma(t))}$ term, and then using Young's inequality one obtains
	\begin{multline}
		\left|\frac{1}{\varepsilon}\hat{a}(G(U_1)- G(U_2),U_1,\invoperator{U_1}\bar{U})\right|\leq C\|\gradg U_1\|_{L^2(\Gamma(t))}^2 \|\bar{U} \|_{U_1}^2 + \frac{\varepsilon}{4}\|\gradg\bar{U}\|_{L^2(\Gamma(t))}^2\\
		+C\|\gradg U_1\|_{L^2(\Gamma(t))} \|\bar{U} \|_{U_1}\|\gradg\bar{U}\|_{L^2(\Gamma(t))}  \log\left( 1 + \frac{C\|\bar{U}\|_{H^2(\Gamma(t))}}{\|\gradg \bar{U}\|_{L^2(\Gamma(t))}} \right)^{\frac{1}{2}}. \label{unique8}
	\end{multline}
	Next using the boundedness of $G(\cdot)$ one finds
	\begin{align}
		\frac{1}{\varepsilon} \hat{a}(G(U_2),\bar{U},\invoperator{U_1}\bar{U}) \leq C \|\bar{U} \|_{U_1}^2 + \frac{\varepsilon}{4}\|\gradg\bar{U}\|_{L^2(\Gamma(t))}^2. \label{unique9}
	\end{align}
	The final, and most troublesome, term to deal with is $\hat{c}(U_2,M(U_1)-M(U_2),\invoperator{U_1}\bar{U})$.
	Using integration by parts, and the Lipschitz continuity of $M(\cdot)$, one has
	\begin{align*}
		\hat{c}(U_2,M(U_1)-M(U_2),\invoperator{U_1}\bar{U}) &= \int_{\Gamma(t)} (M(U_1)- M(U_2)) \gradg \lapg U_2 \cdot \gradg \invoperator{U_1}\bar{U} \\
		& \leq C\|\bar{U}\|_{L^\infty(\Gamma(t))} \|\gradg \lapg U_2\|_{L^2(\Gamma(t))} \|\bar{U}\|_{U_1}.
	\end{align*}
	One then uses Lemma \ref{brezisgallouet} to control the $\|\bar{U}\|_{L^\infty(\Gamma(t))}$ term, as we did above, and Young's inequality to see that
	\begin{multline}
		\varepsilon \hat{c}(U_2,M(U_1)-M(U_2),\invoperator{U_1}\bar{U}) \leq \|\gradg \lapg U_2\|_{L^2(\Gamma(t))}^2 \|\bar{U} \|_{U_1}^2 + \frac{\varepsilon}{4}\|\gradg\bar{U}\|_{L^2(\Gamma(t))}^2\\
		+C\|\gradg \lapg U_2\|_{L^2(\Gamma(t))} \|\bar{U} \|_{U_1}\|\gradg\bar{U}\|_{L^2(\Gamma(t))}  \log\left( 1 + \frac{C\|\bar{U}\|_{H^2(\Gamma(t))}}{\|\gradg \bar{U}\|_{L^2(\Gamma(t))}} \right)^{\frac{1}{2}} \label{unique10}.
	\end{multline}
	Uniqueness follows by using \eqref{unique6}--\eqref{unique10} in \eqref{unique5} and concluding by using Lemma \ref{gronwalltype}, similar to the proof of Theorem \ref{weak strong uniqueness1}.
\end{proof}
\begin{remark}
	\begin{enumerate}
    \item Our assumption that $|U| \leq 1 - \gamma_2$ is known as \emph{separation from the pure phases}, and has been shown to occur on an evolving surface when $M \equiv 1$ --- see \cite{caetano2023regularization}.
    To the authors' knowledge there are no such results of this kind for a degenerate mobility function, although a recent result of this kind has been shown for a positive, variable mobility in \cite{ConGalGat24}.
    Despite this, Lemma \ref{delta control2} shows that $|U| < 1$ almost everywhere, provided that $\Psi$ blows up at $\pm 1$.
    This is the case when $M(r) = (1-r^2)^k$ for $k \geq 2$, as remarked in \cite{elliott1996cahn}.
		\item If one additionally assumes that $F''$ is bounded then the condition that $\gradg \lapg U \in L^2_{L^2}$ actually follows from the assumption that $|U| \leq 1 - \gamma_2$ almost everywhere.
		To see this one can argue as in \cite[Theorem 3]{elliott1996cahn} to show that
		\[ \int_0^T \int_{\Gamma(t)} M^\delta(u^\delta) |\gradg \lapg u^\delta|^2 < C, \]
		for $C$ independent of $\delta$.
		Then if $|U| \leq 1 - \gamma_2$ and $U(t) \in C^0(\Gamma(t))$ for almost all $t \in [0,T]$, one finds that $M(U) > 0$ for almost all $t \in [0,T]$, and so one can use lower semicontinuity to see $\gradg \lapg U \in L^2_{L^2}$.
		As such, this result only requires $U_1$ to be a strong solution --- thus Theorem \ref{weak strong uniqueness2} becomes a weak-strong uniqueness result rather than strong-strong uniqueness.
	\end{enumerate}
\end{remark}

The main take away from this result is that although the weak formulation of \eqref{weakeqn3} requires a stronger class of test functions than \eqref{weakeqn1}, \eqref{weakeqn2} one can still argue along the same lines as in Theorem \ref{weak strong uniqueness1}, given sufficiently strong solutions.
Moreover, this result implies that if one wishes to investigate potential non-uniqueness phenomena for the degenerate Cahn-Hilliard equation, one should look at initial data which attains $\pm 1$ and/or a mobility such that $\lim_{r \rightarrow \pm 1} \Psi(r) < \infty$.
This is the case in the numerical experiments of \cite{barrett1999degenerate}.

\subsection{A deep quench limit of the logarithmic potential}
Here we briefly discuss a special case where the mobility is given by $M(r) = 1-r^2$ and the potential is
\[F_\theta(r) = \frac{\theta}{2} \left((1+r)\log(1+r) + (1-r)\log(1-r)\right) + \frac{1-r^2}{2},\]
for $\theta \in (0,1)$.
In particular we study the limit $\theta \rightarrow 0$, similar to \cite{caetano2021cahn} where the authors show that (for a constant mobility) this limit leads to a unique solution of a variational inequality where the potential is given by the double obstacle potential
\[F_o(r) := \begin{cases}
	\frac{1-r^2}{2}, & r \in [-1,1],\\
	\infty, & |r| > 1.
\end{cases}\]
This deep quench limit has also been studied in \cite{cahn1996cahn,elliott1996cahn}.\\

We denote the solution of the degenerate problem, in the sense of \eqref{weakeqn3}, with a fixed $\theta \in (0,1)$ as $u^\theta$, and note from \eqref{reg energy estimate}, \eqref{reg energy estimate2}, \eqref{reg energy estimate4} that the sequence $u^\theta$ is uniformly bounded (independent of $\theta$) in $H^1_{H^{-1}} \cap L^\infty_{H^1} \cap L^2_{H^2}$.
Denoting the corresponding weak (or weak-$*$) limit by $u$ one can repeat arguments from the previous section to obtain that
\begin{multline*}
	\lim_{\theta \rightarrow 0} m_*(\matdev u^\theta, \phi) + g(u^\theta, \phi) - \varepsilon \hat{c}(u^\theta,M(u^\theta),\phi) - \frac{1}{\varepsilon}\hat{a}(M(u^\theta)F_\theta''(u^\theta),u^\theta,\phi)\\
	= m_*(\matdev u, \phi) + g(u, \phi) - \varepsilon \hat{c}(u,M(u),\phi) + \frac{1}{\varepsilon} \hat{a}(M(u),u,\phi),
\end{multline*}
for some subsequence of $\theta \rightarrow 0$, since 
\[ \hat{a}(M(u^\theta)F_\theta''(u^\theta),u^\theta,\phi) = \theta a(u^\theta, \phi) - \hat{a}(M(u^\theta),u^\theta,\phi) \rightarrow - \hat{a}(M(u^\theta),u^\theta,\phi). \]
From this we conclude that $u$ solves
\[m_*(\matdev u, \phi) + g(u, \phi) - \varepsilon \hat{c}(u,M(u),\phi) + \frac{1}{\varepsilon} \hat{a}(M(u),u,\phi) = 0,\]
for all $\phi \in H^2(\Gamma(t))$, almost all $t \in [0,T]$, and such that $u(0) = u_0$ almost everywhere on $\Gamma_0$.
We note that our uniqueness result does not necessarily mean that the whole sequence $(u^\theta)_\theta$ converges to $u$ as we have not shown the hypotheses of Theorem \ref{weak strong uniqueness2} hold for $u$.

\subsection*{Acknowledgments}
Thomas Sales is supported by the Warwick Mathematics Institute Centre for Doctoral Training, and gratefully acknowledges funding from the University of Warwick and the UK Engineering and Physical Sciences Research Council (Grant number:EP/TS1794X/1).
For the purpose of open access, the author has applied a Creative Commons Attribution (CC BY) licence to any Author Accepted Manuscript version arising from this submission.

\appendix
\section{An integral form of Lemma \ref{gronwalltype}}
In this appendix we outline how one establishes an appropriate integral form of Lemma \ref{gronwalltype}, which is required to justify the proofs of Theorem \ref{weak strong uniqueness1} and Theorem \ref{weak strong uniqueness2}.

\begin{lemma}
	\label{gronwalltype2}
	Let $m_1, m_2, S$ be non-negative functions on $(0,T)$ such that $m_1, S \in L^1(0,T)$ and $m_2 \in L^2(0,T)$, with $S > 0$ a.e. on $(0,T)$.
	Now suppose $f,g$ are non-negative, integrable functions on $(0,T)$, and $f$ is continuous on $[0,T)$, such that
	\begin{align*}
		f(t) + \int_0^t g(s) \, ds \leq \int_0^t m_1(s)f(s) \, ds + \int_0^t m_2(s)\left( f(s)g(s) \log^+\left( \frac{S(s)}{g(s)} \right) \right)^{\frac{1}{2}} \, ds,
	\end{align*}
	holds on $(0,T)$, and $f(0) = 0$.
	Then $f \equiv 0$ on $[0,T)$.
\end{lemma}
\begin{proof}
    The proof of this result is largely the same as that of \cite[Lemma 2.2]{li2016tropical}.
    To begin we recall that for $\sigma \in (0,\infty)$ one has
    \[ \log^+(z) \leq \frac{z^\sigma}{\sigma e}, \text{ for } z \in (0,\infty). \]
    Using this inequality one observes that we have
    \[ f(t) + \int_0^t g(s) \, ds \leq \int_0^t m_1(s)f(s) \, ds + \int_0^t m_2(s) S(s)^\frac{\sigma}{2} g(s)^{\frac{1 - \sigma}{2}} \left( \frac{f(s)}{e \sigma} \right)^\frac{1}{2} \, ds. \]
    Now using Young's inequality we find that for any $\sigma \in (0,1)$ one has, for $a,b \geq 0$,
    \[ ab \leq \frac{1 - \sigma}{2} a^{\frac{2}{1 - \sigma}} + \frac{1 + \sigma}{2} b^{\frac{2}{1 + \sigma}}, \]
    which we use in the above to obtain
    \begin{align*}
        f(t) + \int_0^t g(s) \, ds &\leq \int_0^t m_1(s)f(s) \, ds + \frac{1 - \sigma}{2}\int_0^t g(s) \, ds\\
        &+ \frac{1 + \sigma}{2}\int_0^t m_2(s)^{\frac{2}{1 + \sigma}} S(s)^\frac{\sigma}{1 + \sigma} \left( \frac{f(s)}{e \sigma} \right)^\frac{1}{1 + \sigma} \, ds.
    \end{align*}
    We note that this final integral is well defined by using H\"older's inequality since $f \in L^\infty(0,T), m_2 \in L^2(0,T), S \in L^1(0,T)$.
    Since $\sigma \in (0,1)$ one then finds that
    \begin{align*} f(t) \leq \int_0^t m_1(s)f(s) \, ds + \int_0^t m_2(s)^{\frac{2}{1 + \sigma}} S(s)^\frac{\sigma}{1 + \sigma} \left( \frac{f(s)}{e \sigma} \right)^\frac{1}{1 + \sigma} \, ds.
    \end{align*}
    Now we argue by contradiction and suppose that there is some maximal $t_0 \in [0, T)$ such that $f(t) = 0$ for $t \in [0,t_0]$, and $f(t) > 0$ for $t \in (t_0, T)$.
    Then clearly the above inequality yields that
    \begin{align*} f(t) \leq \int_{t_0}^t m_1(s)f(s) \, ds + \int_{t_0}^t m_2(s)^{\frac{2}{1 + \sigma}} S(s)^\frac{\sigma}{1 + \sigma} \left( \frac{f(s)}{e \sigma} \right)^\frac{1}{1 + \sigma} \, ds,
    \end{align*}
    for $t \geq t_0$.
    Moreover, using the fact that $z \leq \left(\frac{z}{\sigma}\right)^{\frac{1}{1+\sigma}}$ for $z \in (0, \sigma^{-\frac{1}{\sigma}})$ we find that
    \begin{align*}
    f(t) \leq \int_{t_0}^t \left(m_1(s) + m_2(s)^{\frac{2}{1 + \sigma}} S(s)^\frac{\sigma}{1 + \sigma}\right)  \left( \frac{f(s)}{\sigma} \right)^\frac{1}{1 + \sigma} \, ds,
    \end{align*}
    for all $t \in [t_0, t_1]$ where $t_1$ is chosen such that
    \[ \max_{t \in [t_0,t_1]} f(t) \leq \sigma^{-\frac{1}{\sigma}}. \]
    By continuity of $f$ such a $t_1$ exists and moreover we find that $t_1 \rightarrow T$ as $\sigma \rightarrow 0^+$.
    We wish to conclude by using the Bihari-LaSalle inequality \cite{bihari1956generalization}.
    To this end we firstly define the (bijective) function $\Lambda_\sigma: [0,\infty) \rightarrow [0,\infty)$ by
    \[ \Lambda_\sigma(y) := \int_{0}^y \left(\frac{\sigma}{z}\right)^\frac{1}{1+\sigma} \, dz = (1 + \sigma) \sigma^{\frac{-\sigma}{1 + \sigma}} y^\frac{\sigma}{1 + \sigma},\]
    with inverse
    \[ \Lambda_{\sigma}^{-1}(z) := \frac{\sigma}{(1+\sigma)^{\frac{1 + \sigma}{\sigma}}} z^{\frac{1 + \sigma}{\sigma}} .\]
    Now we may apply the Bihari-LaSalle inequality to see that
    \begin{align*} f(t) &\leq \Lambda_\sigma^{-1}\left( \int_{t_0}^t \left(m_1(s) + m_2(s)^{\frac{2}{1 + \sigma}} S(s)^\frac{\sigma}{1 + \sigma}\right) \, ds \right),
    \end{align*}
    for all $t \in [t_0, t_1]$.
    We then notice that
    \[ \lim_{\sigma \rightarrow 0^+} \Lambda_\sigma^{-1}(z) = \begin{cases}
        0, & z \in (0,1],\\
        \infty, & z > 1,
    \end{cases}\]
    and we may choose a $t^* \in (t_0, t_1]$ (independent of $\sigma$) such that
    \[ \int_{t_0}^{t^*} \left(m_1(s) + m_2(s)^{\frac{2}{1 + \sigma}} S(s)^\frac{\sigma}{1 + \sigma}\right) \, ds \leq \frac{1}{2}. \]
    The existence of such a $t^*$ follows from Young's inequality since $m_2 \in L^2(0,T), S \in L^1(0,T)$ so that
    \[\int_{t_0}^{t^*} \left(m_1(s) + m_2(s)^{\frac{2}{1 + \sigma}} S(s)^\frac{\sigma}{1 + \sigma}\right) \, ds \leq \int_{t_0}^{t^*} m_1(s) + (1+\sigma)m_2(s)^2 +  \frac{1}{1+ \sigma}S(s) \, ds.\]
    Then since $\Lambda_\sigma^{-1}(\cdot)$ is increasing we find that
    \[f(t) \leq \Lambda_\sigma^{-1}\left(\frac{1}{2}\right), \text{ for } t \in [t_0, t^*],\]
    and hence we may pass to the limit $\sigma \rightarrow 0^+$ to see that $f(t) = 0$ for $ t \in [t_0, t^*]$.
    This contradicts the maximality of $t_0$, and hence $f \equiv 0$ on $[0,T)$.
    
\end{proof}

\bibliographystyle{acm}
\bibliography{degenmob}
\end{document}